\documentclass{amsart}
\usepackage{amsmath,amssymb,amsthm,amscd,graphicx,float,bm,mathrsfs,mathtools,microtype}
\graphicspath{{figures/}}
\allowdisplaybreaks[4]
\usepackage[pdfpagemode=UseOutlines,bookmarks=true,bookmarksopen=true,
bookmarksopenlevel=0,bookmarksnumbered=true,hypertexnames=false,
colorlinks,linkcolor=blue,citecolor=blue,urlcolor=blue,
pdfstartview=FitV,unicode,breaklinks=true]{hyperref}
\newtheorem{theorem}{Theorem}[section]
\newtheorem{lemma}[theorem]{Lemma}
\newtheorem{proposition}[theorem]{Proposition}
\newtheorem{corollary}[theorem]{Corollary}
\newtheorem{definition}[theorem]{Definition}
\newtheorem{remark}[theorem]{Remark}
\newtheorem{notation}[theorem]{Notation}

\begin{document}
\title[Volume of area-bounded $\mathrm{PGL}_3$-Hitchin moduli space]
{Volume of area-bounded $\mathrm{PGL}_3$-Hitchin moduli space}
\author{Qiongling Li}
\address{Chern Institute of Mathematics and LPMC, Nankai University,
Tianjin 300071, China}
\email{qiongling.li@nankai.edu.cn}
\author{Beibei Liu}
\address{Department of Mathematics, The Ohio State University, Columbus, OH, USA}
\email{liu.11302@osu.edu}
\author{Weixu Su}
\address{School of Mathematics, Sun Yat-sen University, Guangzhou, China}
\email{suwx9@mail.sysu.edu.cn}
\author{Zhe Sun}
\address{School of Mathematical Sciences, University of Science and Technology
of China, 96 Jinzhai Road, 230026 Hefei, Anhui, China}
\email{sunz@ustc.edu.cn}
\keywords{$\mathrm{PGL}_3$-Hitchin moduli space, unmarked convex
$\mathbb{RP}^2$ structures, Goldman symplectic volume, area.}
\subjclass[2010]{Primary 57M50, 58D27, 32G15}
\date{}
\begin{abstract}
For the $\mathrm{PGL}_3$-Hitchin moduli space of unmarked convex
$\mathbb{RP}^2$ structures on the surface $S_{g,m}$ with fixed
boundary data $\mathbf L$ when $m>0$, we investigate the subset whose canonical Hilbert--Busemann area is bounded above by $t$. We show that the Goldman symplectic volume of this subset is at most $C_{g,m}(t+\lambda^2)^{8g-8+3m}$, where $\lambda$ is the largest total-root boundary length for $m>0$ and is set to $0$ for $m=0$.
As a consequence, We obtain the explicit upper bound for the $e^{-s\cdot\mathrm{Area}}$-weighted volume.
\end{abstract}
\maketitle
\enlargethispage{5pt}

\section{Introduction}
Throughout, $S=S_{g,m}$ denotes a connected oriented surface
of genus $g$ with $m$ boundary components and
$\chi(S)=2-2g-m<0$.
In \cite{Mir07a}, Mirzakhani showed that the volume of the moduli space
$\mathcal M_{g,m}(\mathbf L)$ of Riemann surfaces with fixed boundary lengths
$\mathbf L$ with respect to the Weil--Petersson symplectic form is a polynomial
of $\mathbf L$. She obtained this result by showing a recursive formula where
one side consists of the volume of $\mathcal M_{g,m}(\mathbf L)$, while the
other side consists of the volumes of the moduli spaces of Riemann surfaces
obtained by cutting out pairs of pants from $S_{g,m}$ (see \cite{Wri20} for
a survey). In \cite{GS26}, Goncharov and the fourth author find a recursive formula for the exponential volume of the moduli space of ideal hyperbolic structures on a crowned surface by cutting off triangles or trouser legs.

Higher Teichm\"uller theory studies representations of the
fundamental group $\pi_1(S_{g,m})$ with more flexibility, where the isometry
group $\operatorname{PSL}(2,\mathbb R)$ of the holonomy representation of the
hyperbolic surface is replaced by a semisimple Lie group (see \cite{W18} for
a survey). We are looking for an analog of Mirzakhani's result for the
special connected component of the $\operatorname{PGL}(n,\mathbb R)$-representation
variety $\operatorname{Hom}(\pi_1(S_{g,m}),\operatorname{PGL}(n,\mathbb R))/
\operatorname{PGL}(n,\mathbb R)$ modulo the mapping class group. The existence
of such a geometric quantity was predicted by Labourie and McShane in
\cite[page 284]{LM09}, and they indicate that the total volume is infinite
for $n\ge3$, and this result is proved in \cite{CJ23, AKKJ26}. In this paper, we work on the existence of finite geometric
quantities for $n=3$ by imposing a bound on the canonical area or by
integrating an exponentially decaying function of that area.

A \emph{convex $\mathbb{RP}^2$ surface} is a quotient $\Omega/\Gamma$ where
$\Omega\subset\mathbb{RP}^2$ is convex and
$\Gamma\subset\operatorname{PGL}(3,\mathbb R)$ is discrete and acts properly
on $\Omega$. It was initially studied by Kuiper \cite{Ku53,Ku54},
Benz\'ecri \cite{B60}, Kac--Vinberg \cite{KV67} and many others. On the other
hand, a special connected component, called the Hitchin component
$\operatorname{Hit}_n(S_{g,0})$, of the $\operatorname{PGL}(n,\mathbb R)$-representation
variety was found by Hitchin in \cite{Hit92} through a special section of the
Hitchin fibration. In \cite{CG93,G90}, Choi and Goldman proved that the moduli
space of marked convex $\mathbb{RP}^2$ structures on $S_{g,0}$ is exactly
$\operatorname{Hit}_3(S_{g,0})$. In \cite{FG06}, Fock and Goncharov introduced
the notion of positivity to study a special part $\operatorname{Pos}_n(S_{g,m})$
of the representation variety. By \cite[Theorem 1.15]{FG06},
$\operatorname{Pos}_n(S_{g,0})=\operatorname{Hit}_n(S_{g,0})$. For $n=3$,
positivity can be understood in terms of strict convexity.

Let $\operatorname{Pos}_3(S_{g,m})(\mathbf L)$ be the positive representation
variety with fixed boundary simple-root lengths $\mathbf L$. For $m>0$,
Marquis proved in \cite{Mar10} that the moduli space of marked cusped strictly
convex $\mathbb{RP}^2$ structures on $S_{g,m}$ is exactly
$\operatorname{Pos}_3(S_{g,m})(\mathbf0)$. For nonzero boundary data, related
results are established in \cite{Mar12} with respect to the minimal convex
domain. For $m>0$ and $\mathbf L\in\mathbb R_{>0}^{2m}$,
Labourie--McShane's doubling
construction \cite[Section 9]{LM09} provides a canonical closed-surface
positive representation. We are mainly interested in
\[
 \mathcal H(S_{g,m})(\mathbf L)
 :=\operatorname{Pos}_3(S_{g,m})(\mathbf L)/\operatorname{Mod}(S_{g,m}),
\]
the moduli space of unmarked positive convex $\mathbb{RP}^2$ structures with
fixed boundary simple-root lengths. Here $\operatorname{Mod}(S_{g,m})$
consists of orientation-preserving mapping classes preserving each
boundary component setwise, with isotopies subject to the same condition.
When $m=0$, $\mathbf L$ is the empty tuple, and the fixed-boundary
notation denotes the corresponding closed-surface spaces and volumes.

The (Atiyah--Bott--)Goldman symplectic form \cite{AB83,G84} is a natural
mapping-class-group invariant symplectic form on each fixed-boundary leaf.
We use the invariant bilinear form $B(U,V)=\operatorname{tr}(UV)$ to fix its
normalization. For the canonical Hilbert--Busemann area $A(\rho)$ of
Definition~\ref{definition:carea}, write
\begin{gather*}
 \mathcal{AH}^{t}(S_{g,m})(\mathbf L)
 :=\{[\rho]\in\mathcal H(S_{g,m})(\mathbf L): A(\rho)\le t\},\\
 W_{g,m}(t;\mathbf L):=\operatorname{Vol}
      \bigl(\mathcal{AH}^{t}(S_{g,m})(\mathbf L)\bigr).
\end{gather*}
This is a natural subset to consider because the canonical area on the
$3$-Fuchsian locus is the hyperbolic area, hence is constant for fixed
topological type. The main theorem is the following. 

\begin{theorem}[Theorem \ref{thm:main}]\label{thm:intro-main}
Let $\mathbf L=(\ell_1(\alpha),\ell_2(\alpha))_{\alpha\subset \partial S_{g,m}}\in\mathbb R_{>0}^{2m}$. Set
\[
 d=8g-8+3m,\qquad
 \lambda=
 \begin{cases}
 \displaystyle\max_{\alpha\subset\partial S_{g,m}}
      \bigl(\ell_1(\alpha)+\ell_2(\alpha)\bigr),&m>0,\\
 0,&m=0.
 \end{cases}
\]
There is a constant $C_{g,m}>0$, depending only on the topological type and
the stated normalization, such that for every $t\ge0$,
\[
 W_{g,m}(t;\mathbf L)\le C_{g,m}(t+\lambda^2)^d.
\]
\end{theorem}

As a corollary, we show the following exponentially weighted volume is finite.
\begin{corollary}[Theorem \ref{thm:area-weighted-volume}]\label{cor:intro-weight}
For every $s>0$,
\[
 Z_{g,m}(s;\mathbf L):=
 \int_{\mathcal H(S_{g,m})(\mathbf L)}e^{-sA(\rho)}\,d\operatorname{Vol}(\rho)
 <\infty.
\]
More precisely,
\[
 Z_{g,m}(s;\mathbf L)
 =s\int_0^\infty e^{-st}W_{g,m}(t;\mathbf L)\,dt
 \le C_{g,m}\sum_{k=0}^{d}\binom dk
                \lambda^{2(d-k)}\frac{k!}{s^k}.
\]
\end{corollary}

Wolpert's formula \cite{Wol82,Wol83} expresses the Weil--Petersson
symplectic form in Fenchel--Nielsen coordinates. We use its higher-rank
counterpart, the generalized Darboux coordinates of
Sun--Wienhard--Zhang and Sun--Zhang \cite{SWZ20,SZ25}. Then we need to control these parameters with repsect to the area. In \cite{Wol21}, Wolpert provides an example where the area of a local quadrilateral is bounded, but the bulging parameter diverges. However, in our case the area bound controls the two internal coordinates of each pair of pants and the bulging coordinate of each gluing pants curve. The bound for the buglging parameter of each pants curve and the bugling parameter of each pair of pants initially in \cite{S20} is corrected into the bounds in Propositions \ref{prop:closed-bulge} and \ref{prop:internal-pants-area} in this paper. Then one could use McShane identity \cite{HS23} to estimate as in \cite{S20}. We adopt a simpler approach as follows. Lemma \ref{lem:bers}, which basically follows \cite{BPS12}, implies the existence of Bers' constant in $\mathcal{AH}^{t}(S_{g,m})(\mathbf L)$. As a consequence, we can cover $\mathcal{AH}^{t}(S_{g,m})(\mathbf L)$ by a collection of pants curves lengths bounded and twists bounded subsets in Theorem \ref{thm:main}, where the collection runs over the mapping class group orbit of pants decompositions. Integrating each subset proves the main theorem. 
The weighted-volume formula then follows immediately.

\section{Convex \texorpdfstring{$\mathbb{RP}^2$}{RP2} structures on surfaces}
\label{sec:convex-projective}

We recall some preliminaries for investigating the moduli space of unmarked
convex $\mathbb{RP}^2$ structures on the surface, including the convex
$\mathbb{RP}^2$ structures on surfaces, the positive representations and the
projective invariants that are used to parameterize the moduli space.

\subsection{Convex \texorpdfstring{$\mathbb{RP}^2$}{RP2} structure}

\begin{definition}\label{def:rp2-surface}
The $\mathbb{RP}^2$ surface $\Sigma$ is a quotient $\Omega/\Gamma$
diffeomorphic to a smooth surface $S$, where $\Omega$ is a convex domain in
$\mathbb{RP}^2$ and $\Gamma$ is a discrete subgroup of
$\operatorname{PGL}(3,\mathbb R)$ acting properly on $\Omega$.
\end{definition}
Two $\mathbb{RP}^2$ surfaces $\Omega/\Gamma$ and $\Omega'/\Gamma'$ are
equivalent if there is a projective transformation $g\in\operatorname{PGL}(3,\mathbb R)$
such that $(\Omega',\Gamma')=(g\Omega,g\Gamma g^{-1})$.

The $\mathbb{RP}^2$ surface $\Sigma$ is equivalent to a pair $(\rho,f)$:
\begin{itemize}
\item $\rho\colon\pi_1(S)\to\operatorname{PGL}(3,\mathbb R)$ is the holonomy
representation of $\Sigma$ where $\rho(\pi_1(S))=\Gamma$;
\item $f\colon\widetilde S\to\mathbb{RP}^2$ is the developing map where
$f(\widetilde S)=\Omega$.
\end{itemize}
The shape of the domain $\Omega$ is an important feature for the
$\mathbb{RP}^2$ surface.

\begin{definition}\label{def:convexity}
\begin{enumerate}
\item A subset $\Omega$ in $\mathbb{RP}^2$ is convex if the intersection of
$\Omega$ with every line is connected.
\item The convex subset $\Omega$ is properly convex if $\Omega$ is bounded
in $\mathbb R^2\cong\mathbb{RP}^2\setminus\mathbb{RP}^1$ for some hyperplane
$\mathbb{RP}^1$.
\item The properly convex subset $\Omega$ is strictly convex if the boundary
$\partial\Omega$ contains no line segments.
\end{enumerate}
\end{definition}

\begin{definition}\label{def:rp2-structure}
A marked convex $\mathbb{RP}^2$ structure on a smooth surface $S$ is
defined to be a diffeomorphism $h\colon S\to\Sigma$ where $\Sigma$ is a
convex $\mathbb{RP}^2$ surface.
\end{definition}
We say that two marked convex $\mathbb{RP}^2$ structures $(h,\Sigma)$ and
$(h',\Sigma')$ are equivalent if and only if there is a projective equivalence
$g\colon\Sigma\to\Sigma'$ such that $g\circ h$ is isotopic to $h'$.

The unmarked convex $\mathbb{RP}^2$ structure on the smooth surface $S$ is
the pure mapping class group orbit of a marked convex $\mathbb{RP}^2$
structure.

We say that two unmarked convex $\mathbb{RP}^2$ structures $[h,\Sigma]$ and
$[h',\Sigma']$ are equivalent if and only if there is a projective equivalence
$g\colon\Sigma\to\Sigma'$ and an orientation preserving diffeomorphism
$u$ of $S$ which preserves each boundary component setwise such that
$g\circ h\circ u$ is isotopic
to $h'$.

There is a natural Finsler metric on every properly convex domain.
\begin{definition}[Hilbert metric]\label{def:hilbert}
Given a bounded convex domain $\Omega\subset\mathbb R^2\subset\mathbb{RP}^2$, for
any two distinct points $a,b\in\Omega$, let $p_a$ and $p_b$ be the points
at which the straight line $ab$ intersects the boundary of $\Omega$, where
$p_a$ is closer to $a$ and $p_b$ is closer to $b$. Let $|\cdot|$ be the
Euclidean length in $\mathbb R^2$. The Hilbert distance is defined to be
\[
 d_\Omega(a,b)=\frac12\log\left(
 \frac{|a-p_b|}{|b-p_b|}\cdot\frac{|b-p_a|}{|a-p_a|}\right).
\]
The metric defined by the Hilbert distance is called the Hilbert metric.
The Hilbert distance is invariant under projective transformations. Thus
for a convex $\mathbb{RP}^2$ surface $\Omega/\Gamma$, the Hilbert metric on
$\Omega$ descends to the Hilbert metric on $\Omega/\Gamma$.
\end{definition}
In the special case when $\Omega$ is an ellipse for the convex
$\mathbb{RP}^2$ surface $\Omega/\Gamma$, the Hilbert metric on $\Omega$ is
the usual hyperbolic metric on $\Omega$ with respect to the Klein model.

\begin{definition}[Area]\label{def:busemann-area}
For any $(x,v)\in T\Omega$ where $x$ belongs to the convex domain $\Omega$
and $v$ is a tangent vector in $\mathbb R^2$, we denote by $x^+$
($x^-$ respectively) the intersection point of the boundary $\partial\Omega$
and the ray defined by $x$ and $v$ ($-v$ respectively). We define
\[
 |v|_x=\left.\frac{\mathrm d}{\mathrm dt}\right|_{t=0^+}
 d_\Omega(x,x+tv)
 =\frac12\left(\frac1{|x-x^-|}+\frac1{|x-x^+|}\right)|v|.
\]
\begin{itemize}
\item Let $B_x(1)=\{v\in T_x\Omega\mid |v|_x<1\}$.
\item Let $E_B=\pi$ be the Euclidean volume of the open unit ball in $\mathbb R^2$.
\item Let $\operatorname{Leb}$ be the canonical Lebesgue measure of
$\mathbb R^2$ equal to $1$ on the unit square.
\item The density is $h_\Omega(x):=E_B/\operatorname{Leb}(B_x(1))$.
\end{itemize}
For any Borel set $A$ of $\Omega$, the area of $A$ is defined with respect
to the Busemann measure:
\[
 \operatorname{Area}_\Omega(A)=\int_A h_\Omega(x)\,\mathrm d\operatorname{Leb}(x).
\]
\end{definition}

\begin{remark}\label{rem:proper-densities}
There are many other areas defined with respect to different proper
densities \cite{V13}. By a co-compactness result of Benz\'ecri \cite{B60},
any pair of proper densities are comparable. Notably, there is the
Blaschke metric which is Riemannian and uniformly comparable to the
Hilbert metric \cite[Proposition 3.4]{BH13}.
\end{remark}

\subsection{Positive representations}
 In this subsection, we study the convex
$\mathbb{RP}^2$ structure on $S$ from the representation theory point of
view. The holonomy representations of the $\mathbb{RP}^2$ surfaces are
contained in $\operatorname{Hom}(\pi_1(S),\operatorname{PGL}(3,\mathbb R))$.
Modulo the equivalence relation, the representation variety for
$\operatorname{PGL}(3,\mathbb R)$ is
\[
 \operatorname{Hom}(\pi_1(S),\operatorname{PGL}(3,\mathbb R))
 /\operatorname{PGL}(3,\mathbb R),
\]
where $\operatorname{PGL}(3,\mathbb R)$ acts by conjugation. When the
holonomy representation is nice enough, there is a one-to-one correspondence
between the convex $\mathbb{RP}^2$ structure on $S$ up to equivalence and
its holonomy representation up to conjugation.

The $3$-Fuchsian representation is the composition of the discrete faithful
representation from $\pi_1(S)$ to $\operatorname{PSL}(2,\mathbb R)$ and the
irreducible representation $\operatorname{PSL}(2,\mathbb R)$ to
$\operatorname{PGL}(3,\mathbb R)$.

\begin{definition}[Hitchin component, {\cite{Hit92}}]\label{def:hitchin-component}
For $S=S_{g,0}$ being a closed surface of genus $g\ge2$, the
$\operatorname{PGL}(3,\mathbb R)$-Hitchin component $\operatorname{Hit}_3(S)$
is the connected component of
\[
 \operatorname{Hom}(\pi_1(S),\operatorname{PGL}(3,\mathbb R))/\operatorname{PGL}(3,\mathbb R)
\]
that contains all the deformations of $3$-Fuchsian representations.
\end{definition}

\begin{theorem}[{\cite{CG93,G90}}]\label{thm:choi-goldman}
For the integer $g\ge2$, the moduli space of marked strictly convex
$\mathbb{RP}^2$ structures on the surface $S_{g,0}$ is homeomorphic to
$\operatorname{Hit}_3(S_{g,0})$, which is a cell of dimension $16g-16$.
\end{theorem}
For general $n$, some crucial features of the Hitchin component were
unravelled by Fock and Goncharov \cite{FG06} using positivity of matrices
and independently by Labourie \cite{Lab06} using Anosov flows. Then the
notion of Hitchin representation was generalized to positive representation
and Anosov representation in two directions. Both the positive
representations and the Anosov representations are proved to be discrete
and faithful.

We focus on positive representations in this paper. Let us recall the
definition of the positive representations.

\begin{definition}\label{def:flags}
A flag $F$ in $\mathbb R^3$ is a maximal filtration of vector subspaces of
$\mathbb R^3$:
\[
 \{0\}=F^{(0)}\subset F^{(1)}\subset F^{(2)}\subset F^{(3)}=\mathbb R^3,
 \qquad \dim F^{(i)}=i,
\]
denoted by $(F^{(1)},F^{(2)})$. The flag variety is denoted by $\mathcal B$.
Usually, we consider the flag $(F^{(1)},F^{(2)})$ as $(x,X)$ where
$x\in\mathbb{RP}^2$ and $X$ is a line crossing $x$ in $\mathbb{RP}^2$.
\end{definition}
A basis for a flag $F=(F^{(1)},F^{(2)})$ is a basis $(f_1,f_2,f_3)$ for
the vector space $\mathbb R^3$ such that the first $i$ vectors form a basis
for $F^{(i)}$, for $i=1,2$.

\begin{definition}\label{def:generic-position}
We say that the (ordered) $d$-tuple of flags $(F_1,\ldots,F_d)$ are in
generic position if for any mutually distinct $F_a,F_b,F_c$ and
nonnegative integers $i_a,i_b,i_c$ with $i_a+i_b+i_c\le3$, the sum
\[
 F_a^{(i_a)}+F_b^{(i_b)}+F_c^{(i_c)}
\]
is direct.
\end{definition}

\begin{definition}[{\cite[Lemma 9.7]{FG06}}]\label{def:positive-flags}
For the integer $d\ge3$, we say that the $d$-tuple of generic flags
$(F_1,\ldots,F_d)$ in $\mathbb{RP}^2$ is positive if and only if there
exists a strictly convex curve (that bounds a strictly convex domain) such
that the curve is passing the points $(F_1^{(1)},\ldots,F_d^{(1)})$ with
respect to the cyclic order and is tangent to the lines
$(F_1^{(2)},\ldots,F_d^{(2)})$ (see Figure~\ref{fig:positive-flags}). We
define $\operatorname{Conf}_d^+$ to be the space of positive $d$-tuples of
flags up to diagonal projective transformations.
\end{definition}

\begin{figure}[htbp]
\centering
\includegraphics[width=.43\textwidth]{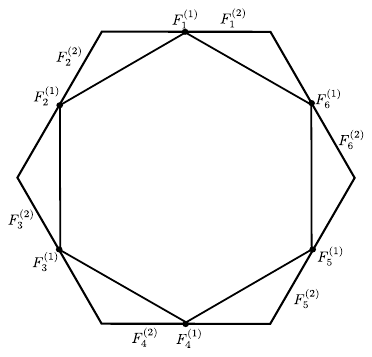}
\caption{\small A positive $6$-tuple of flags.}
\label{fig:positive-flags}
\end{figure}

For any subset $C$ of a circle, we say that the continuous map
$\xi\colon C\to\mathcal B$ is positive if for any cyclically ordered set
$(x_1,\ldots,x_d)$ of $C$ with $d\ge3$, $(\xi(x_1),\ldots,\xi(x_d))$ is a
positive $d$-tuple of flags.

\begin{definition}[Boundary at infinity]\label{def:boundary-infinity}
For each representation
\[
\rho\in\operatorname{Hom}(\pi_1(S_{g,m}),\operatorname{PGL}(3,\mathbb R))/\operatorname{PGL}(3,\mathbb R),
\]
we choose an auxiliary complete hyperbolic structure $\rho_h$ with
geodesic boundary:
\begin{enumerate}
\item For each boundary component $\alpha$, if the monodromy $\rho(\alpha)$
is unipotent, we choose $\rho_h$ such that the boundary $\alpha$ is a cusp.
\item For each boundary component $\alpha$, if the monodromy $\rho(\alpha)$
is not unipotent, we choose $\rho_h$ such that the length of $\alpha$ with
respect to $\rho_h$ is not zero.
\end{enumerate}
Let $\widetilde{(S_{g,m},\rho_h)}$ be the universal cover of
$(S_{g,m},\rho_h)$. Its boundary at infinity, denoted by
$\partial_\infty\pi_1(S_{g,m})$, is the intersection of the absolute
$\partial\mathbb H^2$ with the closure
of $\widetilde{(S_{g,m},\rho_h)}$.
\end{definition}
If $m=0$, $\partial_\infty\pi_1(S_{g,m})$ is homeomorphic to a circle.
If $m\ne0$ and each boundary $\alpha$ of $S_{g,m}$ with respect to
$\rho_h$ is a cusp, the boundary at infinity $\partial_\infty\pi_1(S_{g,m})$
is homeomorphic to a circle. If $m\ne0$ and the length of some geodesic
boundary of $S_{g,m}$ with respect to $\rho_h$ is nonzero, the boundary at
infinity $\partial_\infty\pi_1(S_{g,m})$ is homeomorphic to a Cantor set on
a circle. One can think of $\alpha^+$ and $\alpha^-$ approaching each
other when the length of $\alpha$ with respect to $\rho_h$ approaches zero.

\begin{definition}[Positive representation]\label{def:positive-representation}
The representation $\rho\colon\pi_1(S_{g,m})\to\operatorname{PGL}(3,\mathbb R)$
is positive if there exists a $\rho$-equivariant positive map
$\xi_\rho\colon\partial_\infty\pi_1(S_{g,m})\to\mathcal B$. We denote the
space of positive representations by $\operatorname{Pos}_3(S_{g,m})$.
\end{definition}
Let us recall a nice geometric description of the positive representations.
We restrict to the $\operatorname{PGL}(3,\mathbb R)$ case even though the
following statements are true for any split semisimple algebraic group.

\begin{definition}\label{def:loxodromic}
We say an element in $\operatorname{PGL}(3,\mathbb R)$ is loxodromic if one
of its lifts is conjugate to $\operatorname{diag}(\lambda_1,\lambda_2,\lambda_3)$
where $\lambda_1>\lambda_2>\lambda_3>0$. A matrix is totally positive if all
the minors are positive numbers. An upper triangular matrix is totally
positive if all the minors are positive numbers except the ones that have
to be zero due to the upper triangular condition.
\end{definition}

\begin{theorem}[{\cite[Theorem 2.8]{FG07}}]\label{thm:positive-monodromy}
Given any $\operatorname{PGL}(3,\mathbb R)$-positive representation $\rho$,
for any nontrivial nonperipheral $\gamma\in\pi_1(S_{g,m})$, the monodromy
$\rho(\gamma)$ is conjugate to a totally positive matrix, thus loxodromic.

For any nontrivial peripheral $\gamma\in\pi_1(S_{g,m})$, the monodromy
$\rho(\gamma)$ is conjugate to a totally positive upper triangular matrix.
\end{theorem}
Following the above theorem, we can define the $i$-th length for $i=1,2$.
\begin{definition}[$i$-th length]\label{def:rootlengths}
Given any $\operatorname{PGL}(3,\mathbb R)$-positive representation
$\rho\in\operatorname{Pos}_3(S_{g,m})$, for $i=1,2,3$ and any
$\gamma\in\pi_1(S_{g,m})$, let $\lambda_i(\rho(\gamma))$ be the positive
eigenvalue of $\rho(\gamma)$ where
$\lambda_1(\rho(\gamma))\ge\lambda_2(\rho(\gamma))\ge\lambda_3(\rho(\gamma))>0$.
For $i=1,2$, we define the {\em $i$-th length} (or simple root length) of $\gamma$:
\[
 \ell_i(\gamma):=\ell_i^\rho(\gamma):=
 \log\frac{\lambda_i(\rho(\gamma))}{\lambda_{i+1}(\rho(\gamma))}.
\]
And {\em total-root length}
\[
 \ell(\gamma):=\ell^\rho(\gamma):=\ell_1^\rho(\gamma)+\ell_2^\rho(\gamma)
\]
is twice the Hilbert length of $\gamma$ with respect to $\rho$.
\end{definition}

\begin{definition}\label{def:boundarylengths}
Given any $\operatorname{PGL}(3,\mathbb R)$-positive representation
$\rho\in\operatorname{Pos}_3(S_{g,m})$, let $\alpha_1,\ldots,\alpha_m$ be
the oriented boundary components of the topological surface $S_{g,m}$ such
that $S_{g,m}$ is on the left side of $\alpha_s$ for $s=1,\ldots,m$. Let
\[
 \mathbf L:=\bigl(\ell_1(\alpha_1),\ldots,\ell_1(\alpha_m),
 \ell_2(\alpha_1),\ldots,\ell_2(\alpha_m)\bigr).
\]
We denote the elements in $\operatorname{Pos}_3(S_{g,m})$ with fixed
boundary simple root lengths $\mathbf L$ by
$\operatorname{Pos}_3(S_{g,m})(\mathbf L)$.
\end{definition}
Let us denote $\operatorname{Pos}_3(S_{g,m})(0)$---the collection of
positive representations with unipotent boundary monodromy---by
$\operatorname{Pos}_3^u(S_{g,m})$.

Let $\operatorname{Pos}_3^h(S_{g,m})$ be the collection of positive
representations with loxodromic boundary monodromy. Then
\[
 \operatorname{Pos}_3^h(S_{g,m}):=
 \bigcup_{\mathbf L\in\mathbb R_{>0}^{2m}}
 \operatorname{Pos}_3(S_{g,m})(\mathbf L).
\]
For $m=0$, we use
$\operatorname{Pos}_3^h(S_{g,0})=\operatorname{Pos}_3(S_{g,0})
=\operatorname{Hit}_3(S_{g,0})$.
Let $\operatorname{Pos}'_3(S_{g,m}):=
\operatorname{Pos}_3^h(S_{g,m})\cup\operatorname{Pos}_3^u(S_{g,m})$.

\begin{definition}[Canonical $\rho$-equivariant map]\label{def:canonical-map}
When $m=0$, let $\xi_\rho$ be the unique $\rho$-equivariant positive
boundary map \cite{Lab06}. For the following boundary cases, assume $m>0$.
For any $\operatorname{PGL}(3,\mathbb R)$-positive representation
$\rho\in\operatorname{Pos}_3^h(S_{g,m})$ with loxodromic boundary
monodromy, there is a canonical $\rho$-equivariant map
\[
 \xi_\rho:=(\xi_\rho^1,\xi_\rho^2) \colon\partial_\infty\pi_1(S_{g,m})\to\mathcal B
\]
such that for
any peripheral $\delta\in\pi_1(S_{g,m})$ (by
Theorem~\ref{thm:positive-monodromy}, $\rho(\delta)$ has eigenvectors
$\delta_1,\delta_2,\delta_3$ and the corresponding eigenvalues
$\lambda_1,\lambda_2,\lambda_3$ satisfy $\lambda_1>\lambda_2>\lambda_3>0$),
the eigenvectors $(\delta_3,\delta_2,\delta_1)$
($(\delta_1,\delta_2,\delta_3)$ respectively) form a basis for the flag
$\xi_\rho(\delta^-)$ ($\xi_\rho(\delta^+)$ respectively).

For any $\rho\in\operatorname{Pos}_3^u(S_{g,m})$ with unipotent boundary
monodromy, there is only one choice of $\xi_\rho$
(\cite[Theorem 1.14]{FG06}). We also call $\xi_\rho$ the canonical
$\rho$-equivariant map.
\end{definition}

Similar to Theorem~\ref{thm:choi-goldman}, for
$\operatorname{Pos}_3^u(S_{g,m})$, we have
\begin{theorem}[{\cite{Mar10}}]\label{thm:marquis-cusps}
For the integer $m\ge1$, the deformation space of marked cusped
strictly convex $\mathbb{RP}^2$ structures on the surface $S_{g,m}$ is
homeomorphic to $\operatorname{Pos}_3^u(S_{g,m})$, which is a cell of
dimension $16g-16+6m$.
\end{theorem}
\begin{remark}\label{rem:minimal-domain}
Each $\rho\in\operatorname{Pos}_3(S_{g,m})
\setminus\operatorname{Pos}_3^u(S_{g,m})$
determines a smallest
$\rho(\pi_1(S_{g,m}))$-invariant properly convex open domain
in $\mathbb{RP}^2$; see the proof of
\cite[Proposition 6.16]{Mar12}.
\end{remark}

\begin{remark}\label{rem:doubling}
By \cite[Section 9]{LM09}, a $\operatorname{PGL}(3,\mathbb R)$-representation
for the bordered surface is called Hitchin if it has loxodromic boundary
and is deformed along a path of loxodromic-bordered representations from a
$3$-Fuchsian representation. We denote the space of
$\operatorname{PGL}(3,\mathbb R)$-Hitchin representations up to conjugation
by $\operatorname{Hit}_3(S_{g,m})$. By \cite[Theorem 9.1]{LM09},
$\operatorname{Hit}_3(S_{g,m})\subset\operatorname{Pos}_3^h(S_{g,m})$.
By the canonical doubling construction in
\cite[Definition 9.2.2.3]{LM09} and the positivity of gluing in
\cite[Theorem 7.6]{FG06}, any
$\rho\in\operatorname{Pos}_3^h(S_{g,m})$ with $2g-2+m\ge1$ and $m\ge1$
is glued into a positive representation $d\rho$ for $S_{2g-1+m,0}$ in a
canonical way. By \cite[Theorem 1.15]{FG06},
$\operatorname{Hit}_3(S_{2g-1+m,0})=\operatorname{Pos}_3(S_{2g-1+m,0})$,
so $d\rho$ is a Hitchin representation. Thus $d\rho$ is deformed from a
$3$-Fuchsian representation, and restriction to $S_{g,m}$ induces a
deformation path from the $3$-Fuchsian representation for $S_{g,m}$ to
$\rho$. This proves
$\operatorname{Hit}_3(S_{g,m})=\operatorname{Pos}_3^h(S_{g,m})$.
\end{remark}

\subsection{Projective invariants}
\begin{definition}[Triple ratios]\label{defn:tripleratio}
Consider the triple of flags $(F,G,H)$ in generic position, with bases
\[
 (f_1,f_2,f_3),\quad(g_1,g_2,g_3),\quad(h_1,h_2,h_3).
\]
Then the triple ratio $T(F,G,H)$ is defined by
\[
 T(F,G,H):=
 \frac{\Delta(f^2\wedge g^1)}{\Delta(f^2\wedge h^1)}
 \frac{\Delta(g^2\wedge h^1)}{\Delta(g^2\wedge f^1)}
 \frac{\Delta(h^2\wedge f^1)}{\Delta(h^2\wedge g^1)},
\]
where $w^i:=w_1\wedge\cdots\wedge w_i$, which is
$\operatorname{PGL}(3,\mathbb R)$ invariant. Notice the symmetry
\[
 T(F,G,H)=T(G,H,F)=T(H,F,G).
\]
\end{definition}
Check Figure~\ref{fig:triple-ratio} for a geometric description of the triple ratio.

\begin{figure}[htbp]
\centering
\includegraphics[width=.3\textwidth]{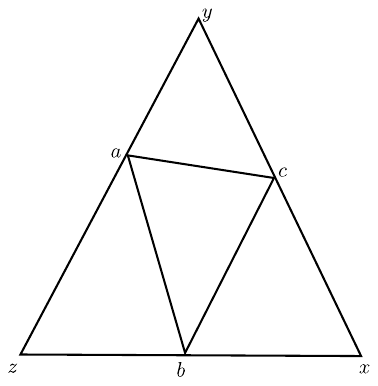}
\caption{\small The flags are $F=(a,yz)$, $G=(b,zx)$, $H=(c,xy)$.
Let $|\cdot|$ be the Euclidean norm. Then the triple ratio
$T(F,G,H)=\frac{|ya|}{|az|}\frac{|zb|}{|bx|}\frac{|xc|}{|cy|}$.
By Ceva's theorem, $T(F,G,H)=1$ if and only if $ax$, $by$ and $cz$
are concurrent.}
\label{fig:triple-ratio}
\end{figure}

\begin{definition}[Edge functions]\label{def:edge-functions}
Let $(X,Y,Z,W)$ be the quadruple of flags in generic position, and choose
their bases
\[
 (x_1,x_2,x_3),\quad(y_1,y_2,y_3),\quad(z_1,z_2,z_3),\quad(w_1,w_2,w_3).
\]
For $i=1,2$, the edge functions are defined to be
\begin{align*}
 D_1(X,Y,Z,W)&:=-\frac{\Delta(x^2\wedge z^1)}{\Delta(x^2\wedge w^1)}
 \cdot\frac{\Delta(x^1\wedge y^1\wedge w^1)}{\Delta(x^1\wedge y^1\wedge z^1)},\\
 D_2(X,Y,Z,W)&:=-\frac{\Delta(y^2\wedge w^1)}{\Delta(y^2\wedge z^1)}
 \cdot\frac{\Delta(x^1\wedge y^1\wedge z^1)}{\Delta(x^1\wedge y^1\wedge w^1)},
\end{align*}
which are $\operatorname{PGL}(3,\mathbb R)$ invariants. Notice the symmetry
\[
 D_1(X,Y,Z,W)=D_2(Y,X,W,Z).
\]
\end{definition}
As shown in Figure~\ref{fig:edge-functions}, the configuration space
$\operatorname{Conf}_4^+$ can be parameterized by the positive numbers
\[
 (A,B,C,D):=\bigl(T(X,W,Y),D_1(X,Y,Z,W),D_2(X,Y,Z,W),T(X,Y,Z)\bigr).
\]
This parametrization depends on the triangulation of the polygon
$(a,c,e,g)$. We can choose the triangulation $\{cg\}$ instead of
$\{ae\}$. Then the parameters are changed into
\[
 (A',B',C',D'):=\bigl(D_2(Z,W,Y,X),T(Z,X,W),T(Z,W,Y),D_1(Z,W,Y,X)\bigr).
\]
Then by \cite[Section 2]{FG07}, we have
\begin{align*}
 (A',B',C',D')=\biggl(&\frac{1+C}{AC(1+B)},\quad
 D\frac{1+C+CA+CAB}{1+B+BD+BDC},\\[-2pt]
 &A\frac{1+B+BD+BDC}{1+C+CA+CAB},\quad
 \frac{1+B}{DB(1+C)}\biggr).
\end{align*}

\begin{figure}[htbp]
\centering
\includegraphics[width=.40\textwidth]{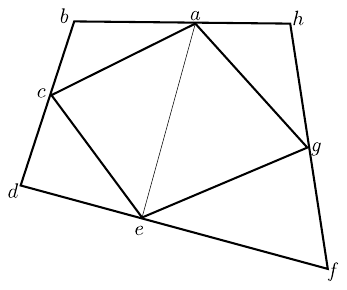}
\caption{\small The flags are $X=(a,bh)$, $W=(c,bd)$, $Y=(e,df)$,
$Z=(g,fh)$. Up to $\operatorname{PGL}(3,\mathbb R)$, the position of
$(X,W,Y)$ is decided by the triple ratio $T(X,W,Y)$. The convention for
cross ratio in $\mathbb{RP}^1$ is
$\operatorname{CR}(\alpha,\beta,\gamma,\delta)
:=\frac{\alpha-\gamma}{\alpha-\delta}\cdot
\frac{\beta-\delta}{\beta-\gamma}$.
We have $D_1(X,Y,Z,W)=-\operatorname{CR}(ab,ae,ag,ac)$ deciding the line
$ag$ and $D_2(X,Y,Z,W)=-\operatorname{CR}(ef,ea,ec,eg)$ deciding the line
$eg$, which fix the point $g$. In the end, the line $hf$ is decided by
the triple ratio $T(X,Y,Z)$.}
\label{fig:edge-functions}
\end{figure}

The space $\operatorname{Conf}_d^+$ can be understood as a map from a
cyclically ordered subset $Q$ of $S^1$ to $\mathcal B$. There is the
$d$-gon $D_d$ with $Q$ as vertices.
The triangulation above is equivalent to the triangulation of the $d$-gon $D_d$.

\begin{definition}\label{def:fg-parameters}
For the integer $d\ge3$, let $(x_1,\ldots,x_d)$ be the cyclically ordered
set $Q$. For any anticlockwise ordered triangle
$\Delta:=(x_i,x_j,x_k)$, we define
\[
 T(\Delta):=T(\xi(x_i),\xi(x_j),\xi(x_k)).
\]
For any edge $e$ with two adjacent anticlockwise ordered triangles
$(x_i,x_j,x_k)$ and $(x_i,x_l,x_j)$, we choose an orientation
$\vec e=(x_i,x_j)$; for $r=1,2$, we define
\[
 D_r(\vec e):=D_r(\xi(x_i),\xi(x_j),\xi(x_k),\xi(x_l)).
\]
\end{definition}
We can use the above parameters to parameterize $\operatorname{Conf}_d^+$.

\begin{proposition}[{\cite{FG06}}]\label{prop:fg-positive-atlas}
For $d\ge3$, given a triangulation $\mathcal T$ of the $d$-gon $D_d$,
let $\Theta$ be the collection of anticlockwise ordered triangles of
$\mathcal T$. Let $E$ be the collection of interior edges of $\mathcal T$.
There exists a real analytic diffeomorphism
$\theta\colon\operatorname{Conf}_d^+\to\mathbb R_{>0}^{3d-8}$,
\[
 \xi\longmapsto
 \bigl((T(\Delta))_{\Delta\in\Theta},
 (D_1(\vec e),D_2(\vec e))_{e\in E}\bigr).
\]
\end{proposition}
Let us recall the ideal triangulation. For more details, check
\cite{Thu79,CB88,PH92,Bon01}.

\begin{definition}[Ideal triangulation]\label{def:ideal-triangulation}
We equip $S_{g,m}$ with a hyperbolic metric $\rho_h$. Choose finitely many
disjoint simple closed geodesics $\mathcal{P}$ (which can be empty). The
ideal triangulation $\mathcal T$ of $S_{g,m}$ (subordinate to $\mathcal{P}$)
is a simple maximal filling geodesic lamination of $(S_{g,m},\rho_h)$
containing $\mathcal{P}$ with finitely many leaves.
\end{definition}
\begin{remark}\label{rem:global-darboux}
When $\mathcal{P}$ cuts the surface into pairs of pants, denote their
collection by $\mathbb{P}$; thus $\mathcal{P}$ is the set of interior
pants curves of $\mathbb{P}$.
For $m=0$ when the ideal triangulation $\mathcal T$ has nonempty
$\mathcal{P}$, Goldman \cite{G90} parameterized the
$\operatorname{PGL}(3,\mathbb R)$-Hitchin component
$\operatorname{Hit}_3(S_{g,0})$. Then Kim \cite{Kim99} provided a global
Darboux coordinate, where some parameters are modified in Choi--Jung--Kim
\cite{CJK21}. Using Fock--Goncharov's parameters \cite{FG06} in
Definition~\ref{def:fg-parameters}, Bonahon and Dreyer \cite{BD14}
parameterized $\operatorname{Hit}_3(S_{g,0})$ with respect to an ideal
triangulation $\mathcal T$ on the closed surface $S_{g,0}$ and a choice
of transverse arcs to $\mathcal{P}$. Then Sun--Wienhard--Zhang
\cite{SWZ20,SZ25} provided a global Darboux coordinate with respect to an
ideal triangulation $\mathcal T$ subordinate to a pants decomposition
$\mathbb{P}$ and a choice of transverse arcs to $\mathcal{P}$. Later on,
we will use the last mentioned global Darboux coordinate system for our
computation.
\end{remark}

\section{Area-bounded moduli spaces}\label{section:3}
In this section we define the canonical area and derive the coordinate
estimates needed below. 

\subsection{Canonical area}
Given a $\operatorname{PGL}(3,\mathbb R)$-positive representation $\rho$
with loxodromic boundary monodromy on $S_{g,m}$, $m\ge1$, let
$\Omega\subset\mathbb{RP}^2$ be the minimal $\rho(\pi_1(S_{g,m}))$-invariant
convex domain. By \cite{Mar12}, the area measured using the Hilbert metric
of this minimal domain is infinite. This is not the metric we use. Indeed,
when $\rho$ is $3$-Fuchsian, the minimal invariant convex domain is the
universal cover of the surface with geodesic boundary, and its own Hilbert
metric is not the hyperbolic metric inherited from the Klein disk.

For $m\ge1$ and $\rho\in\operatorname{Pos}_3^h(S_{g,m})$, consider the canonical
double $d\rho$ of \cite[Definition 9.2.2.3]{LM09}. The
$d\rho(\pi_1(S_{2g-1+m,0}))$-invariant convex domain, denoted by
$\Omega_\rho$, is unique. When $\rho$ is $3$-Fuchsian, $\Omega_\rho$ is a
disk up to projective transformation, so its Hilbert metric is the
hyperbolic metric on the original surface. As a subsurface of the doubled
surface, $S_{g,m}$ has finite area in this metric.

\begin{figure}[htbp]
\centering\includegraphics[width=.3\textwidth]{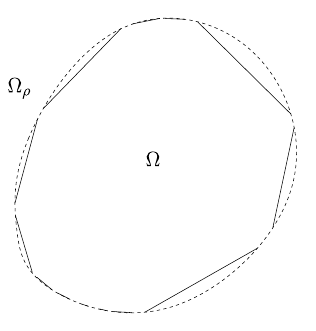}
\caption{\small The $d\rho$-invariant convex domain $\Omega_\rho$ is bounded by the
dotted curve. The minimal $\rho(\pi_1(S_{g,m}))$-invariant convex domain
$\Omega$ is obtained by removing the regions illustrated by the straight
line segments from $\Omega_\rho$.}\label{fig:canonical-domain}
\end{figure}

When $m=0$, let $\Omega_\rho$ be the unique
$\rho(\pi_1(S_{g,0}))$-invariant strictly convex domain. In this case,
$S_{g,0}$ is identified with the closed quotient
$\Omega_\rho/\rho(\pi_1(S_{g,0}))$.

\begin{definition}[Canonical area]\label{definition:carea}
For $\rho\in\operatorname{Pos}_3^h(S_{g,m})$, using the domain
$\Omega_\rho$ specified above, define
\[
 A(\rho):=\operatorname{Area}_{H,\Omega_\rho}(S_{g,m}),
\]
where the area is measured using the Hilbert--Busemann measure
of Definition~\ref{def:busemann-area}.

For $m>0$ and $\rho\in\operatorname{Pos}_3^u(S_{g,m})$, use the same definition
with $\Omega_\rho$ the unique
$\rho(\pi_1(S_{g,m}))$-invariant strictly convex domain.
\end{definition}

\begin{definition}\label{definition:area-locus}
For $t\ge0$ and fixed $\mathbf L\in\mathbb R_{>0}^{2m}$ let
\[
 \operatorname{APos}_3^t(S_{g,m})(\mathbf L)
 =\{\rho\in\operatorname{Pos}_3(S_{g,m})(\mathbf L):A(\rho)\le t\}.
\]
Its quotient by the mapping class group is denoted by
\[
 \mathcal{AH}^t(S_{g,m})(\mathbf L)
 =\operatorname{APos}_3^t(S_{g,m})(\mathbf L)/\operatorname{Mod}(S_{g,m}).
\]
\end{definition}
For $m>0$ with loxodromic boundary monodromy, the projective involution of the
canonical double exchanges its two copies of $S_{g,m}$, so $A(\rho)$
equals half the Hilbert--Busemann area of the closed surface defined
by $d\rho$. The canonical double is invariant under changes of marking
and depends continuously on $\rho$ \cite[Section~9]{LM09}.
For $m>0$ with unipotent boundary monodromy or $m=0$, the area is that of the convex projective structure
itself and is continuous and invariant under changes of marking.
Consequently, $A(\rho)$ is mapping class group invariant and continuous on
$\operatorname{Pos}_3^h(S_{g,m})$, and its sublevel sets are Borel.
The action on the
positive representation space is properly discontinuous
\cite{G90,Lab08}; thus this quotient is well defined. Since each canonical
area is finite, the loci $\mathcal{AH}^t$ exhaust
$\mathcal H(S_{g,m})(\mathbf L)$.

\begin{notation}\label{notation:surface-invariants}
For an ideal triangle $\Delta$ and an oriented ideal edge $\vec e$, use
the canonical equivariant map to define $T(\Delta)(\rho)$ and
$D_i(\vec e)(\rho)$ as in Section~2.3. Equivariance makes the definitions
independent of the chosen lift. Write $t(\Delta)=\log T(\Delta)$ and
$\sigma_i(\vec e)=\log D_i(\vec e)$.
\end{notation}

\begin{proposition}\label{prop:triangle-bound}
If $A(\rho)\le t$, every ideal triangle satisfies
\[
 |\log T(\Delta)(\rho)|\le\sqrt{24t/\pi}.
\]
\end{proposition}
\begin{proof}
By \cite[Proposition~0.3]{AC18} and monotonicity of Hilbert area,
\begin{equation}\label{eq:triangle-area-exact-lower}
 \operatorname{Area}_{H,\Omega_\rho}(\Delta)
 \ge\frac\pi{24}\bigl(\pi^2+(\log T(\Delta))^2\bigr).
\end{equation}
Here Busemann area is $\pi/3$ times the $p$-area of \cite{AC18} on
the triangle bounded by the three supporting lines; see
\cite[Lemma~2.2]{AC18}. The projected triangle has area at most $t$,
which proves the assertion.
\end{proof}

\subsection{Bers' constant for area-bounded subset}
\begin{lemma}[Bers' constant]\label{lem:bers}
Let $\lambda$ be as in Theorem~\ref{thm:intro-main}.
There is a constant $C_S\ge1$, depending only on the topological type,
such that every $[\rho]\in\mathcal{AH}^t(S)(\mathbf L)$ admits a pants
decomposition whose interior pants curves satisfy
\begin{equation}\label{eq:bers}
 \ell(\gamma)\le B(t):=C_S(\sqrt t+\lambda).
\end{equation}
In particular $B(t)$ also bounds every exterior total-root boundary length.
\end{lemma}
\begin{proof}
Let $g_B$ be the Blaschke metric restricted from the canonical
domain. By \cite[Proposition 3.4]{BH13}, it is uniformly comparable
to the Hilbert metric. Thus its area is bounded by a universal
multiple of $t$, and, when $m>0$, its boundary lengths are bounded
by a universal multiple of $\lambda$.

Applying \cite[Proposition 6.3]{BPS12} when $m=0$, and the rescaled
boundary version proved in \cite[Proof of Theorem 6.10]{BPS12} when $m>0$,
gives a pants decomposition whose $g_B$-lengths are at most
$C_S(\sqrt t+\lambda)$.
For the boundary case, after rescaling to make the boundary
lengths at most $1$ and the area at most $-2\pi\chi(S)$,
the area can be increased to $-2\pi\chi(S)$ by enlarging the
metric in the interior, without changing boundary lengths
or decreasing curve lengths.

Replace the pants curves by their Hilbert geodesic representatives.
The metric comparison and $\ell=2\ell_H$ give the required bound
after enlarging $C_S$.
\end{proof}

\subsection{Flows along a simple closed curve}
\label{section:srf}
In \cite{G86}, Goldman introduced flows along a nonperipheral simple closed
curve $\gamma$. Let $(v_1,v_2,v_3)$ be eigenvectors of $\rho(\gamma)$, with
eigenvalues $\lambda_1>\lambda_2>\lambda_3>0$. For $u+v+w=0$ put
\[
 g_r=\operatorname{diag}(e^{ur},e^{vr},e^{wr})
\]
in this eigenbasis. If $\gamma$ separates $S$ into $S_1,S_2$, use the
amalgamated product decomposition to leave $\rho|_{\pi_1(S_1)}$ fixed and
conjugate $\rho|_{\pi_1(S_2)}$ by $g_r$. If $\gamma$ is nonseparating,
leave the representation of the cut surface fixed and multiply the
holonomy of its stable letter by $g_r$. These commuting flows descend
to the representation variety.

We denote $\theta_1,\theta_2$ the simple-root twist flows generated by
\[
 H_1=\operatorname{diag}(2,-1,-1)/3,\qquad
 H_2=\operatorname{diag}(1,1,-2)/3.
\]
By \cite{SWZ20}, their Hamiltonian functions are
\begin{equation}\label{eq:hamiltonians}
 \mu_1=\log\lambda_1=\frac{2\ell_1+\ell_2}{3},\qquad
 \mu_2=-\log\lambda_3=\frac{\ell_1+2\ell_2}{3}.
\end{equation}
The generators of Hamiltonian flows for $\ell_1$ and $\ell_2$ are instead $2H_1-H_2$ and $-H_1+2H_2$.
The \emph{closed-bulging flow} is generated by $H_2-H_1$ and has matrix
\[
 O_r=\operatorname{diag}(e^{-r/3},e^{2r/3},e^{-r/3}).
\]
Thus it changes $(\theta_1,\theta_2)$ by $(-r,r)$.

\subsection{Area control of closed bulging}\label{subsec:closed-bulging}

\begin{figure}[htbp]
\centering\includegraphics[width=0.8\textwidth]{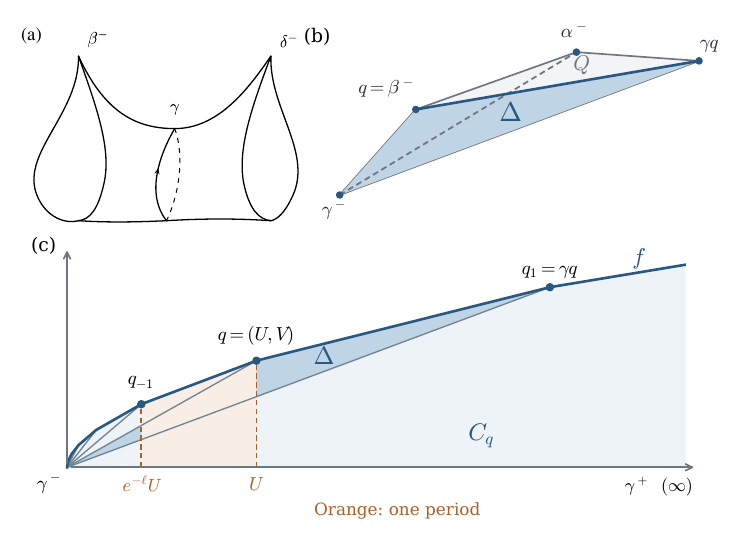}
\caption{\small (a) The two trouser legs adjacent to the pants curve $\gamma$. (b) The quadrilateral $Q$ is fundamental domain of a pair of pants $(\alpha,\beta,\gamma)$, containing a trouser leg $\Delta=(\gamma^-,q,\gamma q)$.
(c) The translates of $\Delta$ fill $C_q$ below the polygonal graph $f$.
Blue highlights $\Delta$; orange marks the other fundamental domain
$e^{-\ell}U<x<U$, $0<y<f(x)$ of the trouser leg by cutting along the perpendicular line from $\beta^-$ to $\gamma$. The point $\gamma^+$
lies at infinity.}\label{fig:closedbul}
\end{figure}

Fix an interior pants curve $\gamma$ and abbreviate
$\ell_i=\ell_i(\gamma)$ and $\ell=\ell(\gamma)=\ell_1+\ell_2$.
Firstly, we choose an affine chart $\{(x,y)\}\cong \{(x,y,1)\}$ with respect to the eigenvectors $(v_1,v_2,v_3)$ of $\gamma$ such that
\[
 \gamma(x,y)=(e^\ell x,e^{\ell_2}y).
\]
As in Figure~\ref{fig:closedbul}(c), we normalize two flags $\xi_\rho(\gamma^-)$, $\xi_\rho(\gamma^+)$ with $\xi_\rho^1(\gamma^-)=(0,0)$, $\xi_\rho^2(\gamma^-)$ being the axis $y=0$, $\xi_\rho^1(\gamma^+)=(\infty,0)$ and $\xi_\rho^2(\gamma^+)$ being the vertical axis at infinity. Then we only have one more vertex to normalize up to $\rm{PGL}_3(\mathbb{R})$ action. Choose the two trouser legs
adjacent to $\gamma$ in Figure~\ref{fig:closedbul}(a) with ideal vertices $\beta^-$ and $\delta^-$. If two boundary components of one pair of pants are identified with $\gamma$, use
the trouser leg adjacent to each boundary component. The left trouser leg is contained in some pair of pants $(\alpha,\beta,\gamma)$ as in Figure \ref{fig:pants-triangulation}. The flip of diagonal resulted in \ref{fig:closedbul}(b) provides a fundamental domain $\Delta:=(\gamma^-,\beta^-,\gamma \beta^-)$ of left trouser leg.
Normalize $\xi_\rho^1(\delta^-)=(1,-1)$ and write the one above $y=0$ as $q=\beta^-=(U,V)$,
with $U,V>0$. 
As in Figure~\ref{fig:closedbul}(c), the translates $\gamma^k\Delta$, $k\in \mathbb{Z}$ tile
$C_q:=\{x>0,\ 0<y<f(x)\}$. The piecewise linear function $f$ joins the points $\gamma^kq=(e^{k\ell}U,e^{k\ell_2}V)$.
Define
\begin{equation}\label{eq:geometric-bulge}
 \beta_\gamma(\rho)=\log V-\frac{\ell_2}{\ell}\log U.
\end{equation}
This expression is unchanged when either point is replaced by its
image under a power of $\gamma$. Under the closed-bulging flow along $\gamma$,
$\beta_\gamma(\rho_r)=\beta_\gamma(\rho)+r$.

\begin{samepage}
\begin{proposition}[Area bound for closed bulging]\label{prop:closed-bulge}
If $A(\rho)\le t$, the quantity $\beta_\gamma$ defined in
\eqref{eq:geometric-bulge} satisfies
\begin{equation}\label{eq:closed-bulge-bound}
 |\beta_\gamma(\rho)|\le \ell+\frac{12t}{\pi \ell}.
\end{equation}
In particular, if $\rho$ and its closed-bulging deformation $\rho_r$ both
have area at most $t$, then
\[
 |r|\le 2\ell+\frac{24t}{\pi \ell}.
\]
\end{proposition}
\end{samepage}
\begin{proof}
By cutting along the perpendicular line from $\beta^-$ to $\gamma$, we get another fundamental domain $e^{-\ell}U<x<U$, $0<y<f(x)$  of the left trouser leg.
Its area in $\Omega_\rho$ is at most $A(\rho)$.
For the opposite boundary orientation, the same argument uses the
triangle $(\gamma^+,q,\gamma q)$. This construction applies on both
sides of $\gamma$.

Replacing $q$ by $\gamma^kq$ for a suitable integer $k$, assume
$1\le U\le e^\ell$.
Let the tangent line
at $(1,-1)$ be $y+Ex+F=0$, where $E,F>0$ and $E+F=1$. Together with
the tangents at $\gamma^-,\gamma^+$ it bounds the outer triangle
\[
 \Omega_\rho\subset\Delta_-:=\{x>0,\ y+Ex+F>0\}.
\]
In the coordinates $(x,y+Ex+F)$ its Hilbert--Busemann density is
\[
 dA_{\Delta_-}=\frac\pi{12}\frac{dx\,dy}{x(y+Ex+F)}.
\]
Here the unit ball in logarithmic coordinates has Euclidean area $12$.
Monotonicity gives
$dA_{\Omega_\rho}\ge dA_{\Delta_-}$. On the polygonal segment from
$(e^{-\ell}U,e^{-\ell_2}V)$ to $(U,V)$,
\[
 f(x)\ge e^{-\ell_2}V,\qquad Ex+F\le EU+F\le U.
\]
By the area comparison for this fundamental domain,
\begin{align*}
 A(\rho)&\ge\frac\pi{12}\int_{e^{-\ell}U}^{U}
        \log\left(1+\frac{f(x)}{Ex+F}\right)\frac{dx}{x}\\
 &\ge\frac\pi{12}\ell\log(1+e^{\beta_\gamma-\ell}),
\end{align*}
where
$e^{-\ell_2}V/U=\exp(\beta_\gamma-(\ell_1/\ell)\log U-\ell_2)\ge e^{\beta_\gamma-\ell}$.
Exchanging the two sides by $(x,y)\mapsto(x/U,-y/V)$ replaces
$\beta_\gamma$ by $-\beta_\gamma$. Therefore
\[
 A(\rho)\ge\frac\pi{12}\ell\log(1+e^{|\beta_\gamma|-\ell})
 \ge\frac\pi{12}\ell(|\beta_\gamma|-\ell),
\]
which proves \eqref{eq:closed-bulge-bound}. The time estimate follows
from $\beta_\gamma(\rho_r)-\beta_\gamma(\rho)=r$.
\end{proof}

\section{Goldman symplectic volume form}
\label{sec:goldman-volume-form}

In this section, using the generalized Darboux coordinate system obtained in
\cite{SWZ20,SZ25}, we express the Goldman symplectic volume form on
$\operatorname{Pos}_3(S_{g,m})(\mathbf L)$ in a natural way.

\subsection{Atiyah--Bott--Goldman symplectic form}For a closed surface $S$ and a compact Lie group $G$,
Atiyah and Bott \cite{AB83} used de Rham cohomology to construct
a natural symplectic form $\omega$ on $\operatorname{Hom}(\pi_1(S),G)/G$. Later on,
Goldman \cite{G84} generalized the symplectic form $\omega$ to noncompact
Lie groups using group cohomology and showed that $\omega$ is a multiple of
the Weil--Petersson symplectic form on the Teichm\"uller space of $S$.
We call $\omega$ the Goldman symplectic form for short.
The Goldman symplectic form has been extended to the case where the
topological surface $S$ has finitely many boundary components with fixed
monodromy conjugacy classes in \cite{AM95,GHJW97} and references therein,
even with marked points on the boundary in \cite{FR98}.
Throughout this section we use the trace normalization
$B(U,V)=\operatorname{tr}(UV)$ on $\mathfrak{sl}_3(\mathbb R)$, as in
\cite[Section~2.5]{SZ25}.

There is a specific natural formula for the Weil--Petersson symplectic
form $\omega_{\mathrm{WP}}$ on Teichm\"uller space. Let
$\mathcal T(S_{g,m})(L_1,\ldots,L_m)$ be the Teichm\"uller space with fixed
boundary lengths. Given a pants decomposition $\mathbb{P}$ of $S_{g,m}$,
with interior pants curves $\mathcal{P}$ and transverse arcs to these
curves, $\mathcal T(S_{g,m})(L_1,\ldots,L_m)$ can be parameterized by the
hyperbolic length functions $l(\gamma)$ and twist functions
$\vartheta(\gamma)$, for $\gamma\in\mathcal{P}$, called the
Fenchel--Nielsen coordinates. In \cite{Wol82,Wol83}, Wolpert provided an
explicit description of the Weil--Petersson symplectic form in terms of
these coordinates, called Wolpert's Magic Formula:
\begin{equation}\label{eq:wolpert}
 \omega_{\mathrm{WP}}
 =\sum_{\gamma\in\mathcal{P}} dl(\gamma)\wedge d\vartheta(\gamma).
\end{equation}
The above formula is crucial in \cite{Mir07a} for computing the volume of
the moduli space
\[
 \mathcal M_{g,m}(L_1,\ldots,L_m):=
 \mathcal T(S_{g,m})(L_1,\ldots,L_m)/\operatorname{Mod}(S_{g,m})
\]
of Riemann surfaces with fixed boundary lengths with respect to
$\omega_{\mathrm{WP}}$.

\subsection{Generalized Wolpert formula}
We use the generalized Wolpert formula
\cite[Corollary~7.11]{SZ25} in the following coordinates.

We specify the ideal triangulation $\mathcal T$ subordinate to a pants
decomposition $\mathbb{P}$ of $S_{g,m}$. Fix an auxiliary hyperbolic
structure $\rho_h$ and orient the $3g-3+m$ interior pants curves in
$\mathcal{P}$. Cutting along their geodesic representatives gives the
$2g-2+m$ pairs of pants in $\mathbb{P}$. For each $P\in\mathbb{P}$, choose
peripheral group elements $\alpha,\beta,\gamma\in\pi_1(P)$ such that
$\alpha\gamma\beta=\operatorname{Id}$ and $P$ lies to the right of
$\alpha,\beta,\gamma$. The inclusion of $P$ into $S_{g,m}$ induces an
inclusion of $\pi_1(P)$ into $\pi_1(S_{g,m})$, so we view these elements as
elements of $\pi_1(S_{g,m})$. Denote by $\delta^+$ and $\delta^-$ the
attracting and repelling fixed points of $\delta\in\pi_1(S_{g,m})$.
The natural projection from the universal cover $\widetilde S_{g,m}$ to
$S_{g,m}$ is denoted by $\pi$. Then $\pi\{\alpha^-,\beta^-\}$ is the
simple geodesic spiralling towards $\alpha$ and $\beta$ opposite to their
respective orientations. In fact, the three simple geodesics
$\pi\{\alpha^-,\beta^-\}$, $\pi\{\beta^-,\gamma^-\}$ and
$\pi\{\gamma^-,\alpha^-\}$ cut $P$ into two ideal triangles
$\pi\{\alpha^-,\beta^-,\gamma^-\}$ and
$\pi\{\alpha^-,\gamma^-,\gamma\cdot\beta^-\}$. The ideal triangulation is
\[
 \mathcal T=\mathcal{P}\cup\bigcup_{P\in\mathbb{P}}
 \bigl\{\pi\{\alpha^-,\beta^-\},\pi\{\beta^-,\gamma^-\},
              \pi\{\gamma^-,\alpha^-\}\bigr\}.
\]

\begin{figure}[htbp]
 \centering
 \includegraphics[width=.6\textwidth]{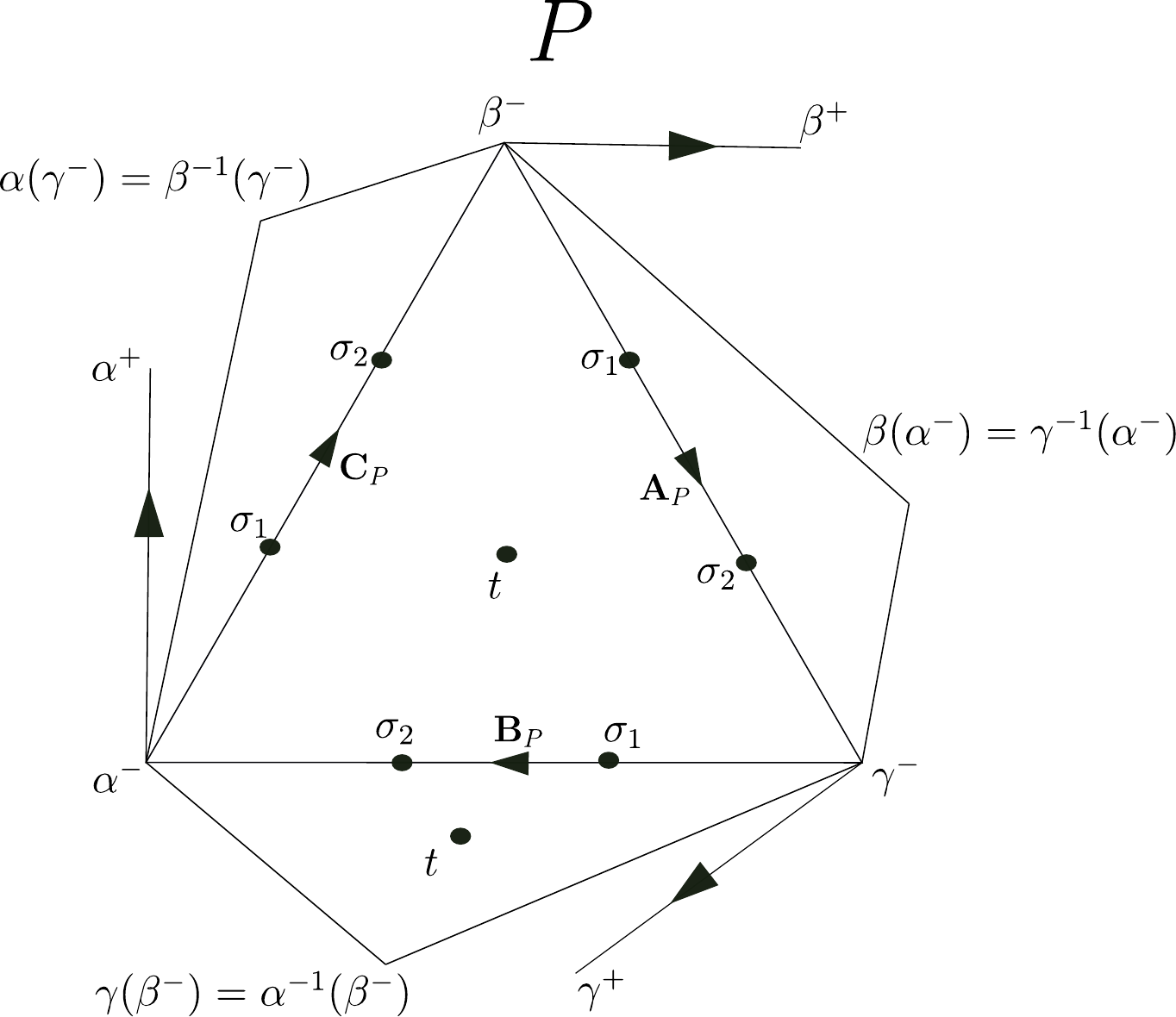}
 \caption{\small $(\gamma^-,\alpha^-,\gamma\cdot\beta^-)$ and
 $(\gamma^-,\beta^-,\alpha^-)$ form a lift of the pair of pants $P$
 with the marking $\alpha\gamma\beta=\operatorname{Id}$.}
 \label{fig:pants-triangulation}
\end{figure}

Let $C_P:=\pi(\alpha^-,\beta^-)$, $A_P:=\pi(\beta^-,\gamma^-)$ and
$B_P:=\pi(\gamma^-,\alpha^-)$. As in
Figure~\ref{fig:pants-triangulation},
$(\gamma^-,\alpha^-,\gamma\cdot\beta^-)$ and
$(\gamma^-,\beta^-,\alpha^-)$ are two adjacent anticlockwise ordered
ideal triangles in the universal cover with common edge
$(\gamma^-,\alpha^-)$. For any
$\rho\in\operatorname{Pos}_3(S_{g,m})(\mathbf L)$, recall the canonical equivariant map $\xi_\rho$. Recall Definition \ref{def:edge-functions} for the edge functions, for
$i=1,2$,
\[
 D_i(B_P)=D_i\bigl(\xi_\rho(\gamma^-),\xi_\rho(\alpha^-),
                  \xi_\rho(\gamma\cdot\beta^-),\xi_\rho(\beta^-)\bigr).
\]
Similarly,
\begin{align*}
 D_i(C_P)&=D_i\bigl(\xi_\rho(\alpha^-),\xi_\rho(\beta^-),
                  \xi_\rho(\alpha\cdot\gamma^-),\xi_\rho(\gamma^-)\bigr),\\
 D_i(A_P)&=D_i\bigl(\xi_\rho(\beta^-),\xi_\rho(\gamma^-),
                  \xi_\rho(\beta\cdot\alpha^-),\xi_\rho(\alpha^-)\bigr).
\end{align*}
Recall Definition \ref{defn:tripleratio} for the triple ratios. Let $\Delta_P:=\pi(\alpha^-,\gamma^-,\beta^-)$ and
$\Delta'_P:=\pi(\alpha^-,\gamma\cdot\beta^-,\gamma^-)$. 
\begin{align*}
 T(\Delta_P)&=T\bigl(\xi_\rho(\alpha^-),\xi_\rho(\gamma^-),
                    \xi_\rho(\beta^-)\bigr),\\
 T(\Delta'_P)&=T\bigl(\xi_\rho(\alpha^-),\xi_\rho(\gamma\cdot\beta^-),
                     \xi_\rho(\gamma^-)\bigr).
\end{align*}

\begin{notation}\label{not:shears}
Using the logarithmic invariants of
Notation~\ref{notation:surface-invariants}, define the internal pants
coordinates
\[
 X_P=\frac1{12}\sum_{e\in\{A_P,B_P,C_P\}}
                  \bigl(\sigma_2(e)-\sigma_1(e)\bigr),\qquad
 Y_P=-t(\Delta_P)+t(\Delta'_P).
\]
\end{notation}

The boundary holonomy identities of \cite{BD14} give the following.
\begin{lemma}\label{lem:pants-boundary-identities}
\begin{align*}
 \ell_1(\alpha)&=\sigma_1(C_P)+\sigma_2(B_P),&
 \ell_2(\alpha)&=\sigma_2(C_P)+t(\Delta_P)+\sigma_1(B_P)+t(\Delta'_P),\\
 \ell_1(\beta)&=\sigma_1(A_P)+\sigma_2(C_P),&
 \ell_2(\beta)&=\sigma_2(A_P)+t(\Delta_P)+\sigma_1(C_P)+t(\Delta'_P),\\
 \ell_1(\gamma)&=\sigma_1(B_P)+\sigma_2(A_P),&
 \ell_2(\gamma)&=\sigma_2(B_P)+t(\Delta_P)+\sigma_1(A_P)+t(\Delta'_P).
\end{align*}
\end{lemma}

Before recalling the symplectic formula, we use
Lemma~\ref{lem:pants-boundary-identities} to control the internal bulging
coordinate of a pair of pants by its area and boundary lengths. All areas
in the following argument are measured in the fixed canonical domain
$\Omega_\rho$ of Definition~3.1. In particular, after cutting the surface
into pairs of pants, we retain the ambient Hilbert metric.

\begin{proposition}[Area bound for internal pants bulging]
\label{prop:internal-pants-area}
Let $P$ be a geodesic pair of pants in the above decomposition, and set
\[
 a_P:=\operatorname{Area}_{H,\Omega_\rho}(P),\qquad
 \Lambda_P:=\sum_{\delta\in\{\alpha,\beta,\gamma\}}\ell(\delta).
\]
Then the internal coordinates of Notation~\ref{not:shears} satisfy
\begin{equation}\label{eq:pants-XY-area}
 |X_P|\le\sqrt{\frac{3a_P}{\pi}}+\frac{\Lambda_P}{6},
 \qquad |Y_P|\le\sqrt{\frac{48a_P}{\pi}}.
\end{equation}
\end{proposition}

\begin{proof}
\emph{Step 1:}
Write
\[
 e_1:=A_P,\quad e_2:=B_P,\quad e_3:=C_P,\qquad
 (\delta_1,\delta_2,\delta_3):=(\alpha,\beta,\gamma),
\]
and put
\[
 p_i:=\sigma_1(e_i)+\sigma_2(e_i),\qquad
 S_P:=t(\Delta_P)+t(\Delta'_P).
\]
Summing the three first-root identities and the three second-root
identities in Lemma~\ref{lem:pants-boundary-identities} gives, respectively,
\[
 \sum_{i=1}^3\ell_1(\delta_i)=\sum_{i=1}^3p_i,\qquad
 \sum_{i=1}^3\ell_2(\delta_i)=\sum_{i=1}^3p_i+3S_P.
\]
Also, adding the two root identities at the boundary $\delta_i$ gives
$\ell(\delta_i)=p_j+p_k+S_P$, where $\{i,j,k\}=\{1,2,3\}$.
It follows that
\begin{equation}\label{eq:pants-fixed-combinations}
\begin{split}
 S_P&=\frac13\sum_{i=1}^3\bigl(\ell_2(\delta_i)-\ell_1(\delta_i)\bigr),\\
 p_i&=\frac{\ell(\delta_j)+\ell(\delta_k)-\ell(\delta_i)-S_P}{2},\\
 r_i:=p_i+S_P&=\frac{\ell(\delta_j)+\ell(\delta_k)-\ell(\delta_i)+S_P}{2}.
\end{split}
\end{equation}
Since the simple root lengths are positive, these formulas imply
\begin{equation}\label{eq:pants-r-bounds}
 |S_P|\le\frac{\Lambda_P}{3},\qquad
 |r_i|\le\frac12\left(\Lambda_P+\frac{\Lambda_P}{3}\right)
          =\frac{2\Lambda_P}{3}.
\end{equation}

\emph{Step 2:}
As in Figure \ref{fig:pants-triangulation}, two adjacent ideal
triangles of each ideal edge $e_i$ form the fundamental domain $Q_i$ of the pair of pants. For example, for $e_2=B_P$, two ideal triangles are
\[
 (\gamma^-,\beta^-,\alpha^-),\qquad
 (\gamma^-,\alpha^-,\gamma\cdot\beta^-).
\]

\begin{figure}[htbp]
 \centering
\includegraphics[width=.7\textwidth]{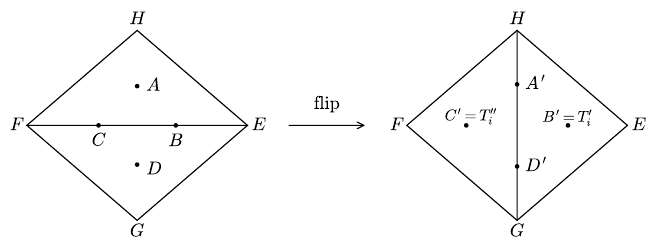}
 \caption{\small Change of coordinates under a flip.}
 \label{fig:flip}
\end{figure}

For each $Q_i$, following the conventions for the quadrilateral after Definition \ref{def:edge-functions}, as in Figure \ref{fig:flip}, let $E,H,F,G$ denote the supporting flags at its four vertices in cyclic order, with $EF$ the original diagonal, and set
\[
 A=T(E,H,F),\quad B=D_1(E,F,G,H),\quad
 C=D_2(E,F,G,H),\quad D=T(E,F,G).
\]
Then
\[ \log B=\sigma_1(e_i),\qquad \log C=\sigma_2(e_i),\qquad
 \log(AD)=S_P.
\]
For the displayed lift one has
$(E,H,F,G)=(\xi_\rho(\gamma^-),\xi_\rho(\beta^-),
\xi_\rho(\alpha^-),\xi_\rho(\gamma\cdot\beta^-))$; the other two leaves
are treated by cyclic relabelling.

Let us flip the diagnoal $EF$ into the other diagonal $HG$.
The two resulting ideal triangles have triple ratios
\[
 T_i'=T(E,H,G),\qquad T_i''=T(H,F,G),\qquad
 u_i=\log T_i',\quad v_i=\log T_i''.
\]
Then by the computation for the flip of diagonal in the quadrilateral after Definition \ref{def:edge-functions}, we get
\begin{equation}\label{eq:flipped-triangle-ratios}
 T_i'=D\frac{1+C+AC+ABC}{1+B+BD+BCD},\qquad
 T_i''=A\frac{1+B+BD+BCD}{1+C+AC+ABC}.
\end{equation}
Put $R=ABCD$, so $\log R=S_P+p_i=r_i$ by
\eqref{eq:pants-fixed-combinations}. Multiplying the first formula by $B$
and grouping its positive terms yields
\begin{equation}\label{eq:weighted-ratio}
 BT_i'=\frac{BD(1+C)+R(1+B)}{BD(1+C)+(1+B)},\qquad
 (BT_i')(CT_i'')=R.
\end{equation}
The first expression is a weighted average of $1$ and $R$, with strictly
positive weights. Hence, if $\eta_i:=\log(BT_i')$, then
\[
 \min(0,r_i)\le\eta_i\le\max(0,r_i),\qquad
 \sigma_1(e_i)+u_i=\eta_i,\qquad
 \sigma_2(e_i)+v_i=r_i-\eta_i.
\]
Subtracting the last two equalities gives
\begin{equation}\label{eq:shear-triangle-difference}
 |\sigma_1(e_i)-\sigma_2(e_i)+u_i-v_i|
 =|2\eta_i-r_i|\le |r_i|.
\end{equation}
\emph{Step 3:}
Applying 
Proposition~\ref{prop:triangle-bound} to the two triangles
of Step~2 gives
\[
 a_P\ge\frac{\pi}{24}(u_i^2+v_i^2).
\]
Consequently,
\[
 |u_i-v_i|
 \le\sqrt{2(u_i^2+v_i^2)}
 \le\sqrt{\frac{48a_P}{\pi}}.
\]
Combining this with \eqref{eq:shear-triangle-difference} gives the
individual edge estimate
\begin{equation}\label{eq:individual-shear-difference}
 |\sigma_1(e_i)-\sigma_2(e_i)|
        \le\sqrt{\frac{48a_P}{\pi}}+|r_i|.
\end{equation}
Finally, by \eqref{eq:pants-r-bounds},
\begin{align*}
 |X_P|&\le\frac1{12}\sum_{i=1}^3
                   |\sigma_1(e_i)-\sigma_2(e_i)|\\
 &\le\sqrt{\frac{3a_P}{\pi}}+\frac1{12}\sum_{i=1}^3|r_i|
 \le\sqrt{\frac{3a_P}{\pi}}+\frac{\Lambda_P}{6}.
\end{align*}
The bound $a_P$ is applied separately to each $Q_i$.

Since $Q_2$ is a fundamentdal domain of the pair of pants $P$, we have
\[
 a_P\ge\frac{\pi}{24}\bigl(t(\Delta_P)^2+t(\Delta'_P)^2\bigr).
\]
Hence
\[
 |Y_P|=|t(\Delta'_P)-t(\Delta_P)|
 \le\sqrt{2\bigl(t(\Delta_P)^2+t(\Delta'_P)^2\bigr)}
 \le\sqrt{\frac{48a_P}{\pi}},
\]
which proves the second inequality in \eqref{eq:pants-XY-area}.
\end{proof}

\subsection{Goldman symplectic volume form}

We are now ready to compute the Goldman symplectic volume form on
$\operatorname{Pos}_3(S_{g,m})(\mathbf L)$.

\begin{proposition}[\cite{SWZ20}]
\label{thm:generalized-wolpert}
For the fixed boundary conjugacy classes of the pair of pants $P$, the symplectic area density on the leaf associated to $P$ is
\[
 2\,|dX_P\,dY_P|,
\]
where $X_P,Y_P$ are as in Notation~\ref{not:shears}.
\end{proposition}

\begin{proposition}[Goldman symplectic volume form]
\label{prop:goldman-volume-form}
Let $\mathbb{P}$ be the set of pairs of pants and $\mathcal{P}$ its set
of interior pants curves. Use $X_P,Y_P$ from Notation~\ref{not:shears}
for $P\in\mathbb{P}$ and $\ell_{i,\gamma}=\ell_i(\gamma)$ for
$\gamma\in\mathcal{P}$. Recall the simple root twist coordinates $\theta_{1,\gamma},\theta_{2,\gamma}$ in
Section \ref{section:srf}. Put $p=|\mathbb{P}|=2g-2+m$,
$n=|\mathcal{P}|=3g-3+m$ and $d=p+2n=8g-8+3m$. The Goldman symplectic volume form is
\begin{equation}\label{eq:goldman-volume-density}
d\operatorname{Vol}:= \frac{|\omega^d|}{d!}
 =\frac{2^p}{3^n}
   \prod_{P\in\mathbb{P}}|dX_P\,dY_P|
   \prod_{\gamma\in\mathcal{P}}|d\ell_{1,\gamma}\,d\ell_{2,\gamma}\,d\theta_{1,\gamma}\,d\theta_{2,\gamma}|.
\end{equation}
\end{proposition}

\begin{proof}
By Proposition~\ref{thm:generalized-wolpert}, the Goldman symplectic volume form contributed by the interior of pairs of pants is
$\prod_{P\in\mathbb{P}}2|dX_P\,dY_P|$. By \eqref{eq:hamiltonians}, the Hamiltonian functions for $\theta_{1,\gamma},\theta_{2,\gamma}$ are $\mu_{1,\gamma}$ and $\mu_{2,\gamma}$. Consequently, the Goldman symplectic volume form is
\[
 2^p\prod_{P\in\mathbb{P}}|dX_P\,dY_P|
 \prod_{\gamma\in\mathcal{P}}
    |d\mu_{1,\gamma}\,d\mu_{2,\gamma}\,d\theta_{1,\gamma}\,d\theta_{2,\gamma}|.
\]
For each $\gamma\in\mathcal{P}$,
\[
 d\mu_{1,\gamma}\wedge d\mu_{2,\gamma}
 =\det\begin{pmatrix}2/3&1/3\\1/3&2/3\end{pmatrix}
      d\ell_{1,\gamma}\wedge d\ell_{2,\gamma}=\frac13\,d\ell_{1,\gamma}\wedge d\ell_{2,\gamma}.
\]
This proves \eqref{eq:goldman-volume-density}.
\end{proof}

\section{Goldman symplectic volume of the area-bounded locus}
\label{sec:volume}
Let $S=S_{g,m}$, with $m\ge0$ and $2g-2+m>0$, and fix the positive
boundary simple-root lengths $\mathbf L$ when $m>0$. For $m=0$,
$\mathbf L$ is empty and $\lambda=0$. Retain $p,n,d$ from
Proposition~\ref{prop:goldman-volume-form}: there are $p$ pairs of pants
and $n$ interior pants curves, and the fixed-boundary symplectic leaf has real
dimension $2d$.
We compute with $d\operatorname{Vol}$
of Proposition~\ref{prop:goldman-volume-form} and set
\[
 W_{g,m}(t;\mathbf L)
 =\operatorname{Vol}\bigl(\mathcal{AH}^t(S)(\mathbf L)\bigr).
\]
All pieces obtained by cutting $S$ retain its ambient canonical Hilbert
metric. Retain $\lambda,C_S$ from Lemma~\ref{lem:bers}, and write
$B=B(t)=C_S(\sqrt t+\lambda)$.

\subsection{The internal pants factors}
For a pair of pants $P$, let $a_P$ be its ambient area and let
$\Lambda_P$ be the sum of its three total-root boundary lengths, counting
each boundary component of $P$ separately. Proposition~\ref{prop:internal-pants-area}
gives
\[
 |X_P|\le\sqrt{\frac{3a_P}{\pi}}+\frac{\Lambda_P}{6},\qquad
 |Y_P|\le\sqrt{\frac{48a_P}{\pi}}.
\]
Since $a_P\le t$, a containing rectangle in the $(X_P,Y_P)$-plane has
Lebesgue area
\begin{equation}
 4\left(\sqrt{\frac{3t}{\pi}}+\frac{\Lambda_P}{6}\right)
       \sqrt{\frac{48t}{\pi}}
 =\frac{48t}{\pi}+8\sqrt{\frac{t}{3\pi}}\,\Lambda_P.
 \label{eq:pants-rectangle}
\end{equation}
All three total-root boundary lengths are at most $B$, so this area is
bounded by
\begin{equation}\label{eq:uniform-pants}
 R(t,B):=\frac{48t}{\pi}
                      +8\sqrt3\,B\sqrt{\frac t\pi}.
\end{equation}

\subsection{The gluing factor and the Dehn-twist quotient}

Fix a marked pants decomposition $\mathbb P$ with interior pants curves $\mathcal P$. For an oriented $\gamma\in\mathcal P$, abbreviate
$\ell_i=\ell_i(\gamma)$ and $\ell=\ell_1+\ell_2$.
Keeping $(1,-1)$ fixed, the twist flow sends the point $q=(U,V)$
of Section~\ref{subsec:closed-bulging} to
\[
 \bigl(e^{\theta_{1,\gamma}+\theta_{2,\gamma}}U,
       e^{\theta_{2,\gamma}}V\bigr).
\]
Thus \eqref{eq:geometric-bulge} gives
\begin{equation}\label{eq:invariant-gluing}
 \beta_\gamma
 =\frac{\ell_1\theta_{2,\gamma}-\ell_2\theta_{1,\gamma}}{\ell}+c_\gamma,
 \qquad
 r_\gamma=\frac{\theta_{1,\gamma}+\theta_{2,\gamma}}{\ell},
\end{equation}
where $c_\gamma$ depends only on the marked structures on the pairs
of pants. This also applies when the two boundary
components identified with $\gamma$ belong to the same pair of pants.

A Dehn twist along $\gamma$ adds $(\ell_1,\ell_2)$ to
$(\theta_{1,\gamma},\theta_{2,\gamma})$, fixing $\beta_\gamma$ and
sending $r_\gamma$ to $r_\gamma+1$.
With the projective structures on the pairs of pants fixed, we have
\[
 |d\theta_{1,\gamma}\,d\theta_{2,\gamma}|
 =\ell\,|d\beta_\gamma\,dr_\gamma|.
\]
Dehn twists along the curves in $\mathcal P$ allow
$0\le r_\gamma<1$ for every $\gamma$ simultaneously.
With the structures on the pairs of pants fixed,
Proposition~\ref{prop:closed-bulge} places the area-bounded locus
inside the product of the intervals
$|\beta_\gamma|\le\ell+12t/(\pi\ell)$.
Integration in these coordinates is therefore bounded, for each
$\gamma$, by
\begin{equation}\label{eq:gluing-factor}
 \ell\int_0^1
   \int_{-\ell-12t/(\pi\ell)}^{\ell+12t/(\pi\ell)}
      d\beta_\gamma\,dr_\gamma
 =2\ell^2+\frac{24t}{\pi}.
\end{equation}

\subsection{Volume calculation}
Let $N_{g,m}$ be the number of mapping-class-group orbits of pants
decompositions with labeled exterior boundary components. These are
encoded by finitely many trivalent graphs with $p$ vertices, $n$ internal
edges and $m$ labeled legs. A coarse bound is
\[
 N_{g,m}\le\frac{(3p)!}{2^n n!}.
\]
This follows by labeling the $3p$ boundary components of the pairs
of pants, choosing the $m$ exterior components in order, and pairing
the remaining $2n$ components.

\begin{theorem}\label{thm:main}
For $B=C_S(\sqrt t+\lambda)$,
\begin{equation}\label{eq:explicit-volume}
 \begin{split}
 W_{g,m}(t;\mathbf L)\le N_{g,m}
 &\left(\frac{96t}{\pi}+16B\sqrt{\frac{3t}{\pi}}\right)^p\cdot\left(\frac{B^4}{6}+\frac{4tB^2}{\pi}\right)^n.
 \end{split}
\end{equation}
In particular,
\begin{equation}\label{eq:polynomial-volume}
 W_{g,m}(t;\mathbf L)\le C_{g,m}(t+\lambda^2)^d,
\end{equation}
where one may take
\begin{equation}\label{eq:volume-constant}
 C_{g,m}=N_{g,m}
 \left(\frac{96}{\pi}+16C_S\sqrt{\frac6\pi}\right)^p
 \left(\frac{2C_S^4}{3}+\frac{8C_S^2}{\pi}\right)^n.
\end{equation}
\end{theorem}
\begin{proof}
Choose representatives $\mathbb{P}^{(1)},\ldots,\mathbb{P}^{(N_{g,m})}$ of the
mapping-class-group orbits of pants decompositions, with exterior boundary
components labeled, and denote their interior pants curves by
$\mathcal{P}^{(1)},\ldots,\mathcal{P}^{(N_{g,m})}$. Fix a bridge system for
each representative and let
\[
 \operatorname{pr}\colon\operatorname{Pos}_3(S)(\mathbf L)
                 \longrightarrow\mathcal H(S)(\mathbf L)
\]
be the quotient map. For each $j$, use the coordinates of
\eqref{eq:invariant-gluing} and define the measurable marked coordinate
region
\begin{equation}\label{eq:covering-regions}
 \begin{split}
 \mathcal F_j(t)=\{\rho\in\operatorname{Pos}_3(S)(\mathbf L):\;&
 A(\rho)\le t,\quad \ell^\rho(\gamma)\le B,\\
 &0\le r_\gamma(\rho)<1\quad\text{for every }\gamma\in\mathcal{P}^{(j)}\}.
 \end{split}
\end{equation}
We claim that
\begin{equation}\label{eq:finite-cover}
 \mathcal{AH}^t(S)(\mathbf L)
       =\bigcup_{j=1}^{N_{g,m}}\operatorname{pr}\bigl(\mathcal F_j(t)\bigr),
\end{equation}
which is basically due to the existence of Bers' constant.
Indeed, given $[\rho]\in\mathcal{AH}^t(S)(\mathbf L)$,
Lemma~\ref{lem:bers} supplies a
pants decomposition $\mathbb{P}_\rho$ whose total-root lengths along the pants curves are
at most $B$. Change the marking so that $\mathbb{P}_\rho$ becomes one
of the representatives $\mathbb{P}^{(j)}$. The resulting marked
representation $\rho'$ has the same area and the corresponding lengths $\ell(\gamma)$. Integer Dehn twists along these disjoint curves then place every
$r_\gamma$ in $[0,1)$, since each twist changes $r_\gamma$ by $1$ and preserves
the area and all lengths $\ell(\gamma)$, $\gamma\in\mathcal P^{(j)}$. Hence $[\rho]$ belongs to the right-hand
side of \eqref{eq:finite-cover}. The reverse inclusion follows directly
from the definition of $\mathcal F_j(t)$.

Since the mapping class group preserves the Goldman symplectic form, the quotient
map preserves the Goldman symplectic volume form away from orbifold singularities.
Then we have
\begin{equation}\label{eq:finite-cover-volume}
 \begin{split}
 W_{g,m}(t;\mathbf L)
 &\le\sum_{j=1}^{N_{g,m}}\operatorname{Vol}
                  \bigl(\operatorname{pr}(\mathcal F_j(t))\bigr)\\
 &\le\sum_{j=1}^{N_{g,m}}\int_{\mathcal F_j(t)}d\operatorname{Vol}.
 \end{split}
\end{equation}
We now bound each integral by the containing coordinate domain of
Sections~5.1--5.2.

For one pants curve $\gamma$, the remaining length integral in
\eqref{eq:gluing-factor} is
\begin{align}
 J(t,B)&=\int_{\substack{\ell_1,\ell_2>0\\\ell_1+\ell_2\le B}}
       \left(2(\ell_1+\ell_2)^2+\frac{24t}{\pi}\right)
                  \,d\ell_1\,d\ell_2\notag\\
 &=\int_0^B\int_0^\ell\left(2\ell^2+\frac{24t}{\pi}\right)
                  \,d\ell_1\,d\ell
   =\frac{B^4}{2}+\frac{12tB^2}{\pi}.
 \label{eq:cuff-integral}
\end{align}
Proposition~\ref{prop:goldman-volume-form} supplies the density factor
$2^p/3^n$. Applying \eqref{eq:uniform-pants} and
\eqref{eq:cuff-integral}, we have
\[
 W_{g,m}(t;\mathbf L)\le N_{g,m}\frac{2^p}{3^n}R(t,B)^pJ(t,B)^n,
\]
which is \eqref{eq:explicit-volume}. 

Since $B\le\sqrt2C_S\sqrt{t+\lambda^2}$,
\begin{align*}
 2R(t,B)&\le\left(\frac{96}{\pi}+16C_S\sqrt{\frac6\pi}\right)(t+\lambda^2),\\
 \frac{J(t,B)}3&\le\left(\frac{2C_S^4}{3}+\frac{8C_S^2}{\pi}\right)(t+\lambda^2)^2.
\end{align*}
Since $p+2n=d$, these inequalities prove
\eqref{eq:polynomial-volume} and \eqref{eq:volume-constant}.
\end{proof}

\subsection{Exponentially weighted Goldman volume}
\label{subsec:area-weighted-volume}

\begin{theorem}
\label{thm:area-weighted-volume}
For every $s>0$, the weighted volume
\[
 Z_{g,m}(s;\mathbf L)
 :=\int_{\mathcal H(S)(\mathbf L)}e^{-sA(\rho)}
             \,d\operatorname{Vol}(\rho)
\]
is strictly positive and finite. Set
 $t_0=\frac{p\pi^3}{12}.$
Then
\begin{equation}\label{eq:weighted-volume-bound}
 Z_{g,m}(s;\mathbf L)
 \le C_{g,m}e^{-st_0}
       \sum_{k=0}^{d}\binom{d}{k}
       (t_0+\lambda^2)^{d-k}\frac{k!}{s^k}.
\end{equation}
\end{theorem}

\begin{proof}
By \eqref{eq:triangle-area-exact-lower},
$\operatorname{Area}_{H,\Omega_\rho}(\Delta)\ge\pi^3/24$
for every ideal triangle of a triangulation subordinate to a
pants decomposition. There are $2p$ such triangles with disjoint
projected interiors. Consequently,
\[
 A(\rho)\ge t_0,\qquad
 W_{g,m}(t;\mathbf L)=0\quad\text{for }0\le t<t_0.
\]

The distribution function of the push-forward of Goldman volume under
$A$ is $W_{g,m}(\,\cdot\,;\mathbf L)$.
The identity
\[
 e^{-sa}=s\int_a^\infty e^{-st}\,dt
\]
implies
\[
 \begin{split}
 Z_{g,m}(s;\mathbf L)
 &=s\int_0^\infty e^{-st}
       \operatorname{Vol}\bigl(\mathcal{AH}^t(S)(\mathbf L)\bigr)
       \,dt\\
 &=s\int_{t_0}^\infty
       e^{-st}W_{g,m}(t;\mathbf L)\,dt.
 \end{split}
\]
Applying \eqref{eq:polynomial-volume} and substituting
$u=t-t_0$, we obtain
\[
 \begin{split}
 Z_{g,m}(s;\mathbf L)
 &\le C_{g,m}s\int_{t_0}^\infty
                  e^{-st}(t+\lambda^2)^d\,dt\\
 &=C_{g,m}e^{-st_0}
       \sum_{k=0}^{d}\binom{d}{k}(t_0+\lambda^2)^{d-k}
       s\int_0^\infty e^{-su}u^k\,du\\
 &=C_{g,m}e^{-st_0}
       \sum_{k=0}^{d}\binom{d}{k}
       (t_0+\lambda^2)^{d-k}\frac{k!}{s^k}.
 \end{split}
\]
This proves finiteness and \eqref{eq:weighted-volume-bound}.
Finally, $\mathcal H(S)(\mathbf L)$ is nonempty and has nonzero Goldman measure.
Since the integrand is strictly positive everywhere,
$Z_{g,m}(s;\mathbf L)>0$.
\end{proof}

\begin{remark}[Area-weighted probability measures]
\label{rem:area-probability-measures}
For each $s>0$, Theorem~\ref{thm:area-weighted-volume} allows us
to define a Borel probability measure on $\mathcal H(S)(\mathbf L)$ by
\begin{equation}\label{eq:area-probability-measure}
 d\mu_{s,\mathbf L}(\rho)
 :=\frac{e^{-sA(\rho)}}{Z_{g,m}(s;\mathbf L)}
       \,d\operatorname{Vol}(\rho).
\end{equation}
Both the canonical area and Goldman volume are invariant under
changes of marking. Thus this family of probability measures
is intrinsically defined on the fixed-boundary moduli space.
\end{remark}

\section*{Acknowledgements}
Q. L. is supported by the National Key R\&D
Program of China No. 2022YFA1006600, the Fundamental Research Funds for the Central Universities, Nankai Zhide Foundation, and the Alexander von Humboldt Foundation. 
W. S. is supported by the NSFC grants 12371076 and 12631004.
 Z. S. is supported by the NSFC grant 12471068 and 12471071. Z. S. would like to thank Fran\c cois Labourie for suggesting the beautiful work of Maryam Mirzakhani to him in 2010. 

\section*{Declaration of AI use}
The authors formulated Proposition \ref{prop:closed-bulge} and Proposition \ref{prop:internal-pants-area} based on the geometric insights, and provided GPT-6 Astra with the key ideas for the proofs to work out the technical details. The authors refined and independently verified the resulting proofs. All mathematical statements and cited references in the manuscript have been independently checked by the authors, who take full responsibility for its content.


\begin{thebibliography}{99}

\bibitem[AB83]{AB83} Michael Atiyah and Raoul Bott,
\textit{The Yang--Mills equations over Riemann surfaces}, Philos. Trans. A \textbf{308} (1983), 523--615.



\bibitem[AC18]{AC18} Ilesanmi Adeboye and Daryl Cooper,
\textit{The area of convex projective surfaces and Fock--Goncharov coordinates}, Journal of Topology and Analysis (2018), 1--13.  	


\bibitem[AKKJ26]{AKKJ26}
Huiseong An, Heejae Kim, Mingook Kim, Seungjin Lee and Hongtaek Jung,
\emph{The Thick Part of the $\mathrm{PSL}_n(\mathbb{R})$-Hitchin-Riemann Moduli Space has Infinite Volume}, preprint (2026), arXiv:2607.06554.



\bibitem[AM95]{AM95} Alekseev, A. Yu, and A. Z. Malkin,
\textit{Symplectic structure of the moduli space of flat connection on a Riemann surface}, Communications in Mathematical Physics \textbf{169} (1995), no. 1, 99-119.

\bibitem[BPS12]{BPS12}
Florent Balacheff, Hugo Parlier, and St\'ephane Sabourau,
\emph{Short loop decompositions of surfaces and the geometry of Jacobians},
Geom. Funct. Anal. \textbf{22} (2012), no.~1, 37--73.  

\bibitem[B60]{B60}
Jean--Paul Benz\'ecri, \emph{Sur les vari\'et\'es localement affines et localement projectives}, Bull. Soc. Math. France, \textbf{88} (1960), 229--332. 



\bibitem[BD14]{BD14}
Francis Bonahon and Guillaume Dreyer, \emph{Parameterizing Hitchin components}, Duke Mathematical Journal \textbf{163} (2014), no.~15, 2935--2975. 

\bibitem[BH13]{BH13}
Yves Benoist, Dominique Hulin \emph{Cubic differentials and finite volume convex projective surfaces}, Geometry And Topology \textbf{17} (2013), 595--620.

\bibitem[Bon01]{Bon01}
Francis Bonahon, \emph{Geodesic laminations on surfaces}, Contemporary Mathematics \textbf{269} (2001), 1--38.


\bibitem[CB88]{CB88}
Andrew J Casson and Steven A Bleiler, \emph{Automorphisms of surfaces after Nielsen and Thurston}, vol. 9, Cambridge University Press, 1988.



\bibitem[CG93]{CG93}
Suhyoung Choi and William~M Goldman, \emph{Convex real projective structures on
  closed surfaces are closed}, Proceedings of the American Mathematical Society
  \textbf{118} (1993), no.~2, 657--661.

\bibitem[CJ23]{CJ23} Suhyoung Choi and Hongtaek Jung, 
\textit{The volumes of the Hitchin--Riemann moduli spaces are infinite}, preprint (2026), arXiv:2305.10093.


\bibitem[CJK21]{CJK21}
Suhyoung Choi, Hongtaek Jung, Hong Chan Kim, \emph{Symplectic coordinates on $\rm{PSL}_3(\mathbb{R})$-Hitchin components}, Pure and applied mathematics quarterly \textbf{16}(2021), no. 5, 1321--1386. arXiv:1901.04651


\bibitem[FG06]{FG06}
Vladimir Fock and Alexander Goncharov, \emph{Moduli spaces of local systems and higher Teichm{\"u}ller theory}, Publications Math\'ematiques de l'Institut des Hautes {\'E}tudes Scientifiques \textbf{103} (2006), no.~1, 1--211. 
  
\bibitem[FG07]{FG07} 
Vladimir Fock and Alexander Goncharov,
\emph{Moduli spaces of convex projective structures on surfaces}, Adv. Math \textbf{208} (2007),  no.~1, 249--273. 

\bibitem[FR98]{FR98} V. V. Fock and A. A. Rosly,
\textit{Poisson structure on moduli of flat connections on Riemann surfaces and $r$-matrix}, Am. Math. Soc. Transl. \textbf{191}. math/9802054 (1998): 67-86.


\bibitem[G84]{G84} William M. Goldman,
\textit{The Symplectic Nature of Fundamental Groups of Surfaces}, Adv. Math. \textbf{54} (1984), 200--225.

\bibitem[G86]{G86} William M. Goldman,
\textit{Invariant functions on Lie groups and Hamiltonian flows of surface group representations}, Invent. Math. \textbf{85} (1986), 263--302.


\bibitem[G90]{G90}
William~M. Goldman, \emph{Convex real projective structures on compact surfaces}, Journal of Differential Geometry
  \textbf{31} (1990), no.~3, 791--845.


\bibitem[GHJW97]{GHJW97} K. Guruprasad, J. Huebschmann, L. Jeffrey, and A. Weinstein,
\textit{Group systems, groupoids, and moduli spaces of parabolic bundles}, Duke Mathematical Journal \textbf{89} (1997), no. 1, 377-412.

\bibitem[GS26]{GS26} Alexander B. Goncharov and Zhe Sun,
\textit{Exponential volumes of moduli spaces of hyperbolic surfaces}, Invent. Math. \textbf{246} (2026), 317--404. 


\bibitem[Hit92]{Hit92}
Nigel J. Hitchin, \emph{Lie groups and Teichm\"uller space}, Topology \textbf{31}(1992), no.~3, 449-473.


\bibitem[HS23]{HS23}
Y. Huang, Z. Sun, \emph{McShane identities for Higher Teichm\"uller theory and the Goncharov-Shen potential}, Mem. Amer. Math. Soc. \textbf{286} (2023), no. 1422. 


\bibitem[Kim99]{Kim99}
Hong Chan Kim, \emph{The symplectic global coordinates on the moduli space of real projective structures}, J. Differential Geom. \textbf{53} (1999), no.~2, 359--401.


\bibitem[Ku53]{Ku53}
Nicolaas Kuiper, \emph{Sur les surfaces localement affines}, Colloq. Inst. G\'eom. Diff., Strasbourg,
CNRS, 1953, 79--87.

\bibitem[Ku54]{Ku54}
Nicolaas Kuiper, \emph{On convex locally projective spaces}, Convegno Int. Geometria Diff., Italy, 1954, 200--213.


\bibitem[KV67]{KV67}
Victor G. Kac, Ernest B. Vinberg, \emph{Quasi-homogeneous cones}, Mathematical Notes of the Academy of Sciences of the USSR \textbf{1}(1967), 231--235.


\bibitem[Lab06]{Lab06}
Fran{\c{c}}ois Labourie, \emph{Anosov flows, surface groups and curves in projective space}, Inventiones mathematicae \textbf{165} (2006), no.~1, 51--114. 

\bibitem[Lab08]{Lab08}
Fran{\c{c}}ois Labourie, \emph{Cross ratios, Anosov representations and the energy functional on Teichm\"uller space}, Ann. Sci. \'Ec. Norm. Sup\'er. (4) \textbf{41} (2008), 437--469.


\bibitem[LM09]{LM09}
Fran{\c{c}}ois Labourie and Gregory McShane, \emph{Cross ratios and identities
  for higher {T}eichm\"uller-{T}hurston theory}, Duke Math. J. \textbf{149}
  (2009), no.~2, 279--345.

\bibitem[Mar10]{Mar10}
Ludovic Marquis, \emph{Espaces des modules des surfaces projectives proprement convexes de volume fini}, Geometry and Topology \textbf{14} (2010), no.~4, 2103--2149.


\bibitem[Mar12]{Mar12}
Ludovic Marquis, \emph{Surface projective proprement convexe de volume fini}, Annales de l'Institut Fourier, \textbf{62} (2012), no.~1, 325--392.


\bibitem[Mir07a]{Mir07a}
Maryam Mirzakhani, \emph{Simple geodesics and {W}eil-{P}etersson volumes of
  moduli spaces of bordered {R}iemann surfaces}, Invent. Math. \textbf{167}
  (2007), no.~1, 179--222.
  

\bibitem[PH92]{PH92}
Robert C Penner and John Harer, \emph{Combinatorics of train tracks}, Princeton University Press, 1992.


\bibitem[S20]{S20}
Zhe Sun, \emph{Volume of the moduli space of unmarked bounded positive convex $\mathbb{RP}^2$ structures}, preprint (2020), arXiv:2001.01295.



\bibitem[SWZ20]{SWZ20}
Zhe Sun, Anna Wienhard, and Tengren Zhang, \emph{Flows on the $\operatorname{PGL}(V)$-Hitchin component}, Geom. Funct. Anal. 30 (2020), no. 2, 588–692. 

\bibitem[SZ25]{SZ25}
Zhe Sun and Tengren Zhang, \emph{The Goldman symplectic form on the $\operatorname{PGL}(V)$-Hitchin component}, Mem. Amer. Math. Soc. 312 (2025), no. 1585. 


\bibitem[Thu79]{Thu79}
William P Thurston, \emph{The geometry and topology of three-manifolds}, Princeton University Princeton, NJ, 1979.

\bibitem[V13]{V13}
Constantin Vernicos, \emph{Asymptotic Volume in Hilbert Geometries}, Indiana Univ. Math. J., \textbf{62} (2013), No. 5, 1431--1441. 


\bibitem[W18]{W18}
Anna Wienhard, \emph{An invitation to higher Teichm\"uller theory}, In Proceedings of the International Congress of Mathematicians: Rio de Janeiro 2018, 1013--1039.

\bibitem[Wol82]{Wol82}
Scott Wolpert, \emph{The Fenchel--Nielsen deformation}, Annals of Mathematics \textbf{115} (1982), no.~3, 501--528.

\bibitem[Wol83]{Wol83}
Scott Wolpert, \emph{On the symplectic geometry of deformations of a hyperbolic surface}, Annals of Mathematics \textbf{117} (1983), no.~2, 207--234.

\bibitem[Wol21]{Wol21}
Scott Wolpert, \emph{The Hilbert area of inscribed triangles and quadrilaterals},  Geometriae Dedicata \textbf{214} (2021), 177--192.


\bibitem[Wri20]{Wri20}
Alex Wright, \emph{A tour through Mirzakhani’s work on moduli spaces of Riemann surfaces}, Bull. Amer. Math. Soc. \textbf{57} (2020), no. 3, 359--408.

\end{thebibliography}
\end{document}